\documentclass[12pt]{amsart}
\usepackage{amsfonts, amssymb, mathrsfs, amsmath}
\usepackage{float, fullpage, verbatim, bbm}
\usepackage[colorlinks=true]{hyperref}
\usepackage{breakurl}
\usepackage{multicol}
\usepackage{tikz}
\usetikzlibrary{calc}
\usepackage[english]{babel}
\usepackage{etoolbox}
\makeatletter
\let\ams@starttoc\@starttoc
\makeatother
\usepackage[parfill]{parskip}
\makeatletter
\let\@starttoc\ams@starttoc
\patchcmd{\@starttoc}{\makeatletter}{\makeatletter\parskip\z@}{}{}
\makeatother

\newcommand\CA{{\mathscr A}} 
\newcommand\CB{{\mathscr B}}
 
\newcommand\CD{{\mathscr D}}
\newcommand\CE{{\mathscr E}}

\newcommand\CAF{{\mathcal {AF}}} 
\newcommand\CIF{{\mathcal {IF}}} 
\newcommand\CRF{{\mathcal {RF}}} 
 
\newcommand\CIFM{{\mathcal {IFM}}} 
 
\newcommand\CRFM{{\mathcal {RFM}}} 
\newcommand\CAFM{{\mathcal {AFM}}}

\newcommand\BBC{{\mathbb C}}
\newcommand\BBK{{\mathbb K}}

\newcommand\BBQ{{\mathbb Q}}

\newcommand\BBZ{{\mathbb Z}}

\newcommand\codim{\operatorname{codim}}

\newcommand\Der{{\operatorname{Der}}}

\DeclareMathOperator{\LMP}{LMP}
\DeclareMathOperator{\GMP}{GMP}

\newcommand\pdeg{\operatorname{pdeg}}

\newcommand\rank{\operatorname{rank}}

\newcommand{\one}{\mathbbm{1}}

\numberwithin{equation}{section}

\theoremstyle{plain}
\newtheorem{lemma}[equation]{Lemma}
\newtheorem{theorem}[equation]{Theorem}

\newtheorem{corollary}[equation]{Corollary}
\newtheorem{proposition}[equation]{Proposition}
\theoremstyle{definition}
\newtheorem{defn}[equation]{Definition}
\newtheorem{remark}[equation]{Remark}

\newtheorem{example}[equation]{Example}

\begin{document}

\title[Free multiderivations of connected subgraph arrangements]
{Free multiderivations of connected subgraph arrangements}

\author[P.~M\"ucksch]{Paul M\"ucksch}
\address
{Institut für Algebra, Zahlentheorie und Diskrete Mathematik, Fakultät
für Mathematik und Physik, Leibniz Universität Hannover, Welfen-
garten 1, D-30167 Hannover, Germany}
\email{muecksch@math.uni-hannover.de}

\author[G.~R\"ohrle]{Gerhard R\"ohrle}
\address
{Fakult\"at f\"ur Mathematik,
	Ruhr-Universit\"at Bochum,
	D-44780 Bochum, Germany}
\email{gerhard.roehrle@rub.de}

\author[S.~Wiesner]{Sven Wiesner}
\address
{Fakult\"at f\"ur Mathematik,
	Ruhr-Universit\"at Bochum,
	D-44780 Bochum, Germany}
\email{sven.wiesner@rub.de}


\allowdisplaybreaks

\begin{abstract}
	Cuntz and Kühne \cite{cuntzkuehne:subgrapharrangements}  introduced the class of connected subgraph arrangements $\CA_G$, depending on a graph $G$,  and classified all graphs $G$ such that the corresponding arrangement $\CA_G$ is free. We extend their result to the multiarrangement case and classify all graphs $G$ for which the corresponding  arrangement $\CA_G$ supports some multiplicity $\mu$ such that the  multiarrangement $(\CA_G,\mu)$ is free. 
\end{abstract}

\maketitle

	
\section{Introduction}

The study of modules of logarithmic vector fields tangent to hyperplane arrangements (in the sequel also referred to as derivation modules)
and in particular the question about their freeness over the coordinate ring of the ambient space
is a classic topic at the crossroads
of discrete geometry and combinatorics, commutative algebra and algebraic geometry initiated in seminal work
by K.\ Saito \cite{Saito80_LogForms} and H.\ Terao \cite{terao:freeI}.
Arrangements with free derivation modules are simply called free arrangements.
Subsequently, G.\ Ziegler \cite{ziegler:multiarrangements} extended the theory by  also taking multiplicities into account.
Such multiarrangements naturally arise from arrangement constructions, i.e.~restrictions to hyperplanes, and the freeness of their
mulitderivation modules is closely linked to the freeness of the original simple arrangement due to fundamental work by Ziegler
and M.\ Yoshinaga \cite{yoshinaga:characterization}.

To decide whether a given (multi-)arrangement is free is a hard problem in general and at its core lies 
Terao's conjecture, asserting that freeness only depends on the underlying combinatorial structure provided by the intersection lattice. To this day the conjecture is wide open.
Therefore, a natural approach in the investigation of such freeness questions is to restrict the view to special
classes of arrangements with additional structure.
A pivotal result along this line is a theorem of Terao  stating that all (complex) reflection arrangements are free \cite{Terao1980_FreeUniRefArr}.
Regarding multiarrangements, another milestone is also due to Terao \cite{terao:free coxeter multiarrangements} establishing
the freeness of certain natural multiplicities on Coxeter arrangements.
This was subsequently used  by Yoshinaga in \cite{yoshinaga:characterization} 
to establish a conjecture of Edelman and Reiner about the freeness of certain deformations of Weyl arrangements.

In the recent work \cite{cuntzkuehne:subgrapharrangements}, M.\ Cuntz and L.\ K\"uhne introduce a new class stemming from simple graphs
which are defined as follows.
\begin{defn}[{\cite[Def.~1.1]{cuntzkuehne:subgrapharrangements}}]
	Let $G=(N, E)$ be an undirected (simple) graph with a set of vertices $N=\{1,\dots,n\}$ and a set of edges $E$.
	The \emph{connected subgraph arrangement} $\CA_G$ in $V=\mathbb{Q}^n$ associated to $G$ is defined as
	\[\CA_G:=\{H_I\mid \emptyset\neq I\subseteq N \text{ if } G[I] \text{ is connected}\},\]
	where $H_I$ is the hyperplane 
	\[H_I= \ker\sum_{i\in I}x_i\]
	and $G[I]$ is the induced subgraph on the vertices $I\subseteq N$. 
\end{defn}
Cuntz and K\"uhne gave a complete characterization of free arrangements within this class (see Theorem \ref{theorem: FreeConnectedSubgraphArrangements}).
In view of their work and because of the importance of the broader perspective on freeness within the framework of multiarrangements,
the natural question arises if one can describe certain free multiplicities extending their classification.
This is the aim of this paper. Our first main result is as follows.

\begin{theorem}[{Corollary \ref{coro:GraphsWithFreeMultiplicities}}]
	\label{thm:Main1}
	Let $G$ be a connected graph. There exists a multiplicity $\mu$ such that the connected subgraph multiarrangement $(\CA_G,\mu)$ is free if and only if $G$ is $G_1$, $G_2$, a path-graph, a cycle-graph, an almost-path-graph, or a path-with-triangle-graph (see Definition \ref{definition: cuntz kühne graph families}
	and Figure \ref{fig:G1_8}).
\end{theorem}

A natural and important question stemming from the study of free multi-Coxeter arrangements by Terao \cite{terao:free coxeter multiarrangements}
is about the existence of constant free multiplicities.
Our next result gives the following answer for connected subgraph arrangements.

\begin{theorem}[{Proposition \ref{prop:G_1_2ConstNotFree} and Corollary \ref{coro:ConstMultFree}}]
	\label{thm:Main2}
	Let $G$ be a connected graph with at least three vertices and $\mu\equiv c > 1$ is
	a constant multiplicity on $\CA_G$. 
	Then $(\CA_G,\mu)$ is free if and only if $G$ is a path-graph $P_n$  or $G$ is the cycle-graph $C_3$ and $c=3$.
\end{theorem}

To classify free multiplicities in general is a notoriously hard problem. Complete answers are known only in very
few cases. E.g.\ characterizing the free multiplicities even on braid arrangements of rank greater than $3$
is an unresolved problem, cf.\ \cite{dipasquale: free A3 classification}.
Nonetheless, in our study it turns out that for a certain subarrangement $\CD$ of  $\CA_{C_3}$ which is obtained from $\CA_{C_3}$ by deleting
a single hyperplane, the task is feasible. This constitutes our third theorem (Theorem \ref{Theorem: AllDMultiplicities}) which was previously obtained by M.\ DiPasquale and M.\ Wakefield in \cite{dipasquale: X3 moduli freeness} by different methods.

\begin{theorem}[{\cite{dipasquale: X3 moduli freeness}}]
	\label{thm:Main3}
	Let $(\mathscr{D},\mu)$ be the multiarrangement $$Q(\mathscr{D},\mu)=x_1^ax_2^bx_3^c(x_1+x_2)^d(x_1+x_3)^e(x_2+x_3)^f$$ with $a,b,c,d,e,f \geq 1$. Then
	$(\mathscr{D},\mu)$ is free if and only if  $\mu=(2k,2k,2k,1,1,1)$ for some $k\in\mathbb{N}_{\geq 1}$ where the order of the hyperplanes is as above.
\end{theorem}

To conclude, our last chief result describes certain families of free multiplicities on $\CA_{C_3}$.
\begin{theorem}[{Proposition \ref{prop:SpecialFreeMultC_3}}]
	\label{thm:Main4}
	Let $(\CA_{C_3},\mu)$ be the multiarrangement $$Q(\CA_{C_3},\mu)=x_1^ax_2^bx_3^c(x_1+x_2)^d(x_1+x_3)^e(x_2+x_3)^f(x_1+x_2+x_3)^g$$ with multiplicity $\mu=(a,b,c,d,e,f,g), (d\geq e\geq f)$. Suppose that either 
	\begin{enumerate}
		\item[(i)] $\mu=(k,k,k,r,1,1,k)$, where $1\leq k \leq 3$ and $r\geq 1$, or
		\item[(ii)] $\mu=(k,k,k,r,1,1,k)$, where $k>3$ and $r\geq 2k-5$.
	\end{enumerate}
	Then $(\CA_{C_3},\mu)$ is free. Moreover, if $r\geq 2k$, then $(\CA_{C_3},\mu)$ is  inductively free with exponents $(r,2k+1,2k+1)$.
    For $\mu = (3,3,3,1,1,1,3)$, $(\CA_{C_3},\mu)$ is not inductively free.
\end{theorem}

For the notion of inductive freeness, see Definition \ref{def:indfree}.

The paper is organized as follows.
In \S \ref{sect:prelim} we review the required notions from the theory of free multiarrangements
and connected subgraph arrangements.
The next section 
presents the arguments leading to a proof of Theorem \ref{thm:Main1}.
In \S \ref{sect:FreeConstantMultiplicities}, we investigate the natural question which connected subgraph arrangements support free constant
multiplicities leading to our second main 
Theorem \ref{thm:Main2}.
Along the way,  in \S
\ref{subsect:nonfreeag} we present a non-computational proof 
of the non-freeness of the connected subgraph arrangements stemming from $G_1$ up to $G_8$ from Figure \ref{fig:G1_8}. This was derived in 
 \cite{cuntzkuehne:subgrapharrangements} by means of
elaborate computer calculations.
In our final section 
we study free multiplicities on two special arrangements,
culminating in Theorems \ref{thm:Main3} and \ref{thm:Main4}.
In the appendices 
we list some tables containing data which we computed and utilised for our 
proofs.


\section{Recollections and Preliminaries}
\label{sect:prelim}

\subsection{Hyperplane arrangements, multiarrangements and their freeness}
\label{ssect:hyper}

Let $\BBK$ be a field of characteristic $0$ and let
$V = \BBK^\ell$ 
be an $\ell$-dimensional $\BBK$-vector space.
A \emph{hyperplane arrangement} is a pair
$(\CA, V)$, where $\CA$ is a finite collection of hyperplanes in $V$.
Usually, we simply write $\CA$ in place of $(\CA, V)$.
We write $|\CA|$ for the number of hyperplanes in $\CA$.
The empty arrangement in $V$ is denoted by $\Phi_\ell$.

The \emph{lattice} $L(\CA)$ of $\CA$ is the set of subspaces of $V$ of
the form $H_1\cap \ldots \cap H_i$ where $\{ H_1, \ldots, H_i\}$ is a subset
of $\CA$. 
For $X \in L(\CA)$, we have two associated arrangements, 
firstly
$\CA_X :=\{H \in \CA \mid X \subseteq H\} \subseteq \CA$,
the \emph{localization of $\CA$ at $X$}, 
and secondly, 
the \emph{restriction of $\CA$ to $X$}, $(\CA^X,X)$, where 
$\CA^X := \{ X \cap H \mid H \in \CA \setminus \CA_X\}$.
Note that $V$ belongs to $L(\CA)$
as the intersection of the empty 
collection of hyperplanes and $\CA^V = \CA$. 
The lattice $L(\CA)$ is a partially ordered set by reverse inclusion:
$X \le Y$ provided $Y \subseteq X$ for $X,Y \in L(\CA)$. The rank function on $L(\CA)$ is given by $\rank(X)=\codim X$ for $X\in L(\CA)$ and we define $L(\CA)_k=\{X\in L(\CA)\mid \rank(X)=k\}$.

Let $S = S[V] \cong \BBK[x_1, \ldots , x_\ell]$ be the coordinate ring of the ambient space with its usual grading by polynomial degree:
$S = \oplus_{p \in \BBZ}S_p$, where $S_p$ is the $\BBK$-subspace of polynomials of degree $p$ and $S_p = 0$ in case $p < 0$.

A \emph{multiarrangement}  is a pair
$(\CA, \mu)$ consisting of a hyperplane arrangement $\CA$ and a 
\emph{multiplicity} function
$\mu : \CA \to \BBZ_{\ge 0}$ associating 
to each hyperplane $H$ in $\CA$ a non-negative integer $\mu(H)$.
The \emph{order} of the multiarrangement $(\CA, \mu)$ 
is defined by 
$|\mu| := 
\sum_{H \in \CA} \mu(H)$.
For a multiarrangement $(\CA, \mu)$, the underlying 
arrangement $\CA$ is sometimes called the associated 
\emph{simple} arrangement, and so $(\CA, \mu)$ itself is  
simple if and only if $\mu(H) = 1$ for each $H \in \CA$. We sometimes denote this simple multiplicity by $\mu = \one$.
For two multiplicities $\mu, \mu'$ on an arrangement $\CA$ we sometimes write $\mu \geq \mu'$ if $\mu(H) \geq \mu'(H)$ for all $H \in \CA$.

Let $(\CA, \mu)$ be a multiarrangement in $V$ and let 
$X \in L(\CA)$. The 
\emph{localization of $(\CA, \mu)$ at $X$} is defined to be $(\CA_X, \mu_X)$,
where $\mu_X = \mu |_{\CA_X}$.

The \emph{defining polynomial} $Q(\CA, \mu)$ 
of the multiarrangement $(\CA, \mu)$ is given by 
\[
Q(\CA, \mu) := \prod_{H \in \CA} \alpha_H^{\mu(H)},
\] 
a polynomial of degree $|\mu|$ in $S$.

Let $\Der(S)$ be the $S$-module of algebraic $\BBK$-derivations of $S$.
Using the $\BBZ$-grading on $S$, $\Der(S)$ becomes a graded $S$-module.
For $i = 1, \ldots, \ell$, 
let $D_i := \partial/\partial x_i$ be the usual derivation of $S$.
Then $D_1, \ldots, D_\ell$ is an $S$-basis of $\Der(S)$.
We say that $\theta \in \Der(S)$ is 
\emph{homogeneous of polynomial degree p}
provided 
$\theta = \sum_{i=1}^\ell f_i D_i$, 
where $f_i$ is either $0$ or homogeneous of degree $p$
for each $1 \le i \le \ell$.
In this case we write $\pdeg \theta = p$.

The \emph{module of $\CA$-derivations} of $(\CA, \mu)$ is 
defined as 
\[
D(\CA, \mu) := \{\theta \in \Der(S) \mid \theta(\alpha_H) \in \alpha_H^{\mu(H)} S 
\text{ for each } H \in \CA\}.
\]
We say that $(\CA, \mu)$ is \emph{free} if 
$D(\CA, \mu)$ is a free $S$-module -- a notion introduced by G.\ Ziegler
\cite[Def.~6]{ziegler:multiarrangements}, extending K.\ Saito's \cite{Saito80_LogForms} and H.\ Terao's \cite{terao:freeI} freeness for simple arrangements.

The derivation module
$D(\CA, \mu)$ is a $\BBZ$-graded $S$-module and 
thus, if $(\CA, \mu)$ is free, there is a 
homogeneous basis $\theta_1, \ldots, \theta_\ell$ of $D(\CA, \mu)$.
The multiset of the unique polynomial degrees $\pdeg \theta_i$ 
forms the set of \emph{exponents} of the free multiarrangement $(\CA, \mu)$
and is denoted by $\exp (\CA, \mu)$. Freeness is a local property for multiarrangements.

\begin{theorem}[{\cite[Prop.~1.7]{abenuidanumata:signedeliminable}}]\label{theorem: Multilocalizations are free}
	For $U\subseteq V$ a subspace, the localization $(\CA_U,\mu_U)$ at $U$ is free provided $(\CA,\mu)$ is free. 
\end{theorem}

If $\mu = \one$, we also write $D(\CA)$ for $D(\CA,\one)$ and say that \emph{$\CA$ is free} if $(\CA,\one)$ is.
In any case the Euler derivation, $\theta_E = \sum_{i=1}^\ell x_i D_i$ belongs to $ D(\CA)$ and so if $\CA$ is free, $1$ is always an exponent of $\CA$, \cite[Prop.~4.27]{orlikterao:arrangements}.

Let $\CA$ be an arrangement and $H\in\CA$. Define a multiplicity $\kappa$ on $\CA^H$ by $$\kappa(X)=\vert\CA_X\vert-1,\quad X\in\CA^H$$ and call the pair $(\CA^H,\kappa)$ the \emph{Ziegler restriction} of $\CA$ to $H$. Ziegler showed the following connection between the freeness of $\CA$
a simple arrangement 
and the freeness of any of its Ziegler restrictions.

\begin{theorem}[{\cite[Thm.~11]{ziegler:multiarrangements}}]\label{theorem: ziegler restriction}
	Suppose $\CA$ is free with exponents $(1,e_2,\dots,e_\ell)$. Let $H$ be a hyperplane in $\CA$, then $(\CA^H,\kappa)$ is free with exponents $(e_2,\dots,e_\ell)$. 
\end{theorem}

Ziegler also generalizes Saito's criterion to the multiarrangement case.

\begin{theorem}[{\cite[Thm.~8]{ziegler:multiarrangements}}, Saito’s criterion]\label{theorem: saitos criterion}
	Let $(\CA,\mu)$ be a multiarrangement in $\mathbb{K}^n$ and for each  $1\leq i\leq \ell$ let
 $$\theta_i=\sum_{j=1}^\ell p_{ij}D_j\in D(\CA,\mu).$$
 Then $\{\theta_1,\theta_2,\dots,\theta_\ell\}$ is a basis for $D(\CA,\mu)$ if and only if $\det(p_{ij})=c\cdot Q(\CA,\mu)$ for some $c\in\mathbb{K}\setminus \{0\}$.
\end{theorem}

An arrangement $\CA$ is called \emph{locally free along} $H$ for some $H\in\CA$, if $\CA_X$ is a free arrangement for all subspaces $X$ of $V$ with $\{0\}\neq X \subset H$. The following result is due to Yoshinaga.

\begin{theorem}[{\cite[Thm.~2.2]{yoshinaga:characterization}}]
	\label{Theorem:localFreeness}
	Suppose $\CA$ is an arrangement in $\BBK^\ell$ with $\ell > 3$.
	Then $\CA$ is free with exponents $(1 = e_1,e_2, \dots,e_\ell)$ if and only if it is locally free along $H$ for some $H\in \CA$ and $(\CA^H,\kappa)$ is free with exponents $(e_2,\dots,e_\ell)$.
\end{theorem}

Theorem \ref{Theorem:localFreeness} 
readily implies the following.

\begin{corollary}
	\label{Cor: LocalFreeChaPolFactor}
	Suppose $\CA$ is an arrangement in $\BBK^\ell$ with $\ell > 3$. 
	Let $\CA$ be locally free along $H\in \CA$. If $\CA_X$ fails to be free for some $X\in L(\CA)$ with $X\not\subset H$, then $(\CA^H,\kappa)$ is not free.
\end{corollary}

\begin{proof}
	Suppose $(\CA^H,\kappa)$ is free. Then so is $\CA$, by Theorem \ref{Theorem:localFreeness}, and so is $\CA_X$ for any $X \in L(\CA)$, 
	by Theorem \ref{theorem: Multilocalizations are free} (with simple multiplicities), a contradiction. 
\end{proof}

Recall that an $\ell$-arrangement $\CA$ with $\rank(\CA)=r$ is called \emph{generic} if $\vert\CA_X\vert=k$ for all $X\in L(\CA)_k$ and $1\leq k\leq r-1$. An arrangement $\CA$ is called \emph{totally non-free} if $(\CA,\mu)$ fails to be free for any multiplicity $\mu \geq \one$. The following result states that generic arrangements are totally non-free.

\begin{theorem}[{\cite[Prop.~4.1]{yoshinaga:extendable}}]\label{theorem: Yoshinaga generic not free}
	Let $\CA$ be a generic arrangement in $V$ with $\ell=\dim V\geq 3$. If $\vert\CA\vert>\ell$, then $(\CA,\mu)$ is totally non-free.
\end{theorem}

The following construction from \cite{abeteraowakefield:euler} leads to
an extension of Terao's seminal Addition-Deletion Theorem \cite{terao:freeI} for free simple arrangements
to multiarrangements.

\begin{defn}
	\label{def:Euler}
	Let $(\CA, \mu) \ne \Phi_\ell$ be a multiarrangement. Fix $H_0 = \ker(\alpha_0)$ in $\CA$.
	We define the \emph{deletion}  $(\CA', \mu')$ and \emph{restriction} $(\CA'', \mu^*)$
	of $(\CA, \mu)$ with respect to $H_0$ as follows.
	If $\mu(H_0) = 1$, then set $\CA' = \CA \setminus \{H_0\}$
	and define $\mu'(H) = \mu(H)$ for all $H \in \CA'$.
	If $\mu(H_0) > 1$, then set $\CA' = \CA$
	and define $\mu'(H_0) = \mu(H_0)-1$ and
	$\mu'(H) = \mu(H)$ for all $H \ne H_0$.
	
	Let $\CA'' = \{ H \cap H_0 \mid H \in \CA \setminus \{H_0\}\ \}$.
	The \emph{Euler multiplicity} $\mu^*$ of $\CA''$ is defined as follows.
	Let $Y \in \CA''$. Since the localization $\CA_Y$ is of rank $2$, the
	multiarrangement $(\CA_Y, \mu_Y)$ is free, 
	\cite[Cor.~7]{ziegler:multiarrangements}. 
	According to 
	\cite[Prop.~2.1]{abeteraowakefield:euler},
	the module of derivations 
	$D(\CA_Y, \mu_Y)$ admits a particular homogeneous basis
	$\{\theta_Y, \psi_Y, D_3, \ldots, D_\ell\}$,
	such that $\theta_Y \notin \alpha_0 \Der(S)$
	and $\psi_Y \in \alpha_0 \Der(S)$,
	where $H_0 = \ker \alpha_0$.
	Then on $Y$ the Euler multiplicity $\mu^*$ is defined
	to be $\mu^*(Y) = \pdeg \theta_Y$.
	
	Often, 
	$(\CA, \mu), (\CA', \mu')$ and $(\CA'', \mu^*)$ 
	is referred to as the \emph{triple} of 
	$(\CA, \mu)$ with respect to $H_0$. 
\end{defn}

The following observation is immediate from Definition \ref{def:Euler}.

\begin{remark}
	\label{rem:euler}
	Let $(\CA,\mu)$ be a multiarrangement and $H_0\in \CA$ and let $(\CA'',\mu^*)$ be the Euler multiplicity corresponding to $H_0$. For $X\in\CA''$ the value $\mu^*(X)$ is simply the common value of the non-zero exponent of $(\CA_X,\mu_X)$ and $(\CA'_X,\mu_X')$.	
\end{remark} 

It is useful to be able to determine the Euler multiplicity explicitly in some 
relevant instances. For that purpose, 
we recall parts of \cite[Prop.~4.1]{abeteraowakefield:euler}.

\begin{proposition}
	\label{ATWEulerProp}
	Let $(\CA,\mu)$ be a multiarrangement. Let $X\in\CA^{H_0}$, $m_0=\mu(H_0)$, $k=\vert \CA_X\vert$, and $m_1=\max\{\mu(H)\mid H\in \CA_X\backslash\{H_0\}\}$.
	\begin{itemize}
		\item[(i)] If $k=2$, then $\mu^*(X)=m_1$.
		\item[(ii)] If $\vert \mu_X\vert\leq 2k-1$ and $m_0>1$, then $\mu^*(X)=k-1$.
		\item[(iii)] If $\mu_X\equiv 2$, then $\mu^*(X)=k$.
	\end{itemize}
\end{proposition}

We can now state the general Addition-Deletion Theorem for free multiarrangements.

\begin{theorem}
	[{\cite[Thm.~0.8]{abeteraowakefield:euler}}
	Addition-Deletion Theorem for Multiarrangements]
	\label{thm:add-del}
	Suppose that $(\CA, \mu) \ne \Phi_\ell$.
	Fix $H_0$ in $\CA$ and 
	let  $(\CA, \mu), (\CA', \mu')$ and  $(\CA'', \mu^*)$ be the triple with respect to $H_0$. 
	Then any  two of the following statements imply the third:
	\begin{itemize}
		\item[(i)] $(\CA, \mu)$ is free with $\exp (\CA, \mu) = (e_1, \ldots , e_{\ell -1}, b_\ell)$;
		\item[(ii)] $(\CA', \mu')$ is free with $\exp (\CA', \mu') = (e_1, \ldots , e_{\ell -1}, b_\ell-1)$;
		\item[(iii)] $(\CA'', \mu^*)$ is free with $\exp (\CA'', \mu^*) = (e_1, \ldots , e_{\ell -1})$.
	\end{itemize}
\end{theorem}

\begin{remark}
	\label{rem:AddDelThm_multi_simple}
	If we restrict to simple multiplicities, Theorem \ref{thm:add-del} 
	recovers Terao's original Addition-Deletion Theorem 
 from \cite{terao:freeI}. As in the simple case if $(\CA, \mu)$ and $(\CA',\mu')$ are free, then all three statements of Theorem \ref{thm:add-del} hold (see \cite[Thm.~0.4]{abeteraowakefield:euler}).
\end{remark}

Analogous to the simple case, Theorem \ref{thm:add-del} motivates 
the notion of inductive freeness. 

\begin{defn}[{\cite[Def.~0.9]{abeteraowakefield:euler}}]
	\label{def:indfree}
	The class $\CIFM$ of \emph{inductively free} multiarrangements 
	is the smallest class of multiarrangements subject to
	\begin{itemize}
		\item[(i)] $\Phi_\ell \in \CIFM$ for each $\ell \ge 0$;
		\item[(ii)] for a multiarrangement $(\CA, \mu)$, if there exists a hyperplane $H_0 \in \CA$ such that both
		$(\CA', \mu')$ and $(\CA'', \mu^*)$ belong to $\CIFM$, and $\exp (\CA'', \mu^*) \subseteq \exp (\CA', \mu')$, 
		then $(\CA, \mu)$ also belongs to $\CIFM$.
	\end{itemize}
\end{defn}

Similarly, we also have the notion of recursive freeness for multiarrangements.
\begin{defn}[{\cite[Def.~2.21]{hogeroehrleschauenburg:free}}]
	\label{def:recfree-mult}
	The class $\CRFM$ of \emph{recursively free} multiarrangements 
	is the smallest class of multiarrangements subject to
	\begin{itemize}
		\item[(i)] $\Phi_\ell \in \CRFM$ for each $\ell \ge 0$;
		\item[(ii)] for a multiarrangement $(\CA, \mu)$, if there exists a hyperplane $H_0 \in \CA$ such that both
		$(\CA', \mu')$ and $(\CA'', \mu^*)$ belong to $\CRFM$, and $\exp (\CA'', \mu^*) \subseteq \exp (\CA', \mu')$, 
		then $(\CA, \mu)$ also belongs to $\CRFM$.
		\item[(iii)] for a multiarrangement $(\CA, \mu)$, if there exists a hyperplane $H_0 \in \CA$ such that both
		$(\CA, \mu)$ and $(\CA'', \mu^*)$ belong to $\CRFM$, and $\exp (\CA'', \mu^*) \subseteq \exp (\CA, \mu)$, 
		then $(\CA', \mu')$ also belongs to $\CRFM$.
	\end{itemize}
\end{defn}

Moreover, we consider the following weaker notion.

\begin{defn}[{\cite[Def.~3.3]{hogeroehrle:zieglerII}}]
	\label{def:multaddfree}
	The multiarrangement $(\CA, \mu)$ is said to be  
	\emph{additively free} if there is a free filtration of multiplicities $\mu_i$ on $\CA$,
	\[
	\mu_0 < \mu_1 < \cdots < \mu_n = \mu,
	\]
	i.e., where each $(\CA, \mu_i)$ is free with $|\mu_i| = i$.
	Denote this class by $\CAFM$. In particular, $(\CA, \mu_0) = \Phi_\ell$. 
\end{defn}

\begin{remark}
	\label{rem:rank2indfree}
	As for simple arrangements, if $r(\CA) \le 2$,
	then $(\CA, \mu)$  is inductively free,  
	\cite[Cor.~7]{ziegler:multiarrangements}.
	Also, $\CIFM \subseteq \CRFM$  and $\CIFM \subseteq \CAFM$ and all inclusions are strict.
	
	By Remark \ref{rem:AddDelThm_multi_simple},
	restricting to the subclasses of simple arrangements within the above defined
	 classes, we recover the classes of inductively free $\CIF$ (cf.~\cite[Def.~4.53]{orlikterao:arrangements}), additively free $\CAF$ (cf.~\cite[Def.~1.6]{abe:sf}),
	 and recursively free $\CRF$ (cf.~\cite[Def.~4.60]{orlikterao:arrangements}) simple arrangements.	 
\end{remark}

In general, it is a hard problem to compute the exponents of a free multiarrangement even in the rank $2$ case. We require the following result by Wakamiko.

\begin{theorem}[{\cite[Thm.~1.5]{wakamiko:2arrangements}}]\label{WakaTheo}
	Let $\CA=\{H_1,H_2,H_3\}$ be a $2$-arrangement of three lines, $\mu$ a multiplicity on $\CA$ with $\mu(H_i)=k_i,(i=1,2,3)$ and $\exp(\CA,\mu)=(e_1,e_2)$. Assume that $k_3\geq\max\{k_1,k_2\}$ and let $k=k_1+k_2+k_3$.
	\begin{itemize}
		\item[(i)] If $k_3< k_1+k_2-1$, then  
		\begin{equation*}
			\vert e_1-e_2\vert =
			\begin{cases}
				0 & \text{if $k$ is even,}\\
				1 & \text{if $k$ is odd.}\\
			\end{cases}       
		\end{equation*}
		\item[(ii)] If $k_3 \geq k_1+k_2-1$, then $\exp(\CA,\mu)=(k_1+k_2,k_3)$.
	\end{itemize}
\end{theorem}


\subsection{Connected subgraph arrangements}
We summarize the important definitions and results from \cite{cuntzkuehne:subgrapharrangements}. 
We start by recalling the definition of a connected subgraph arrangement. In \cite{cuntzkuehne:subgrapharrangements} the authors introduced this concept for an arbitrary underlying field. Here we restrict ourselves to arrangements over $\BBQ$.

\begin{defn}[{\cite[Def.~1.1]{cuntzkuehne:subgrapharrangements}}]
	Let $G=(N,E)$ be an undirected (simple) graph with vertex set $N=[n] =\{1,\dots,n\}$ and edges $E$. We define the \emph{connected subgraph arrangement} $\CA_G$ in $\BBQ^n$ as
	\[\CA_G:=\{H_I\mid \emptyset\neq I\subseteq N \text{ if } G[I] \text{ is connected}\},\]
	where $H_I$ is the hyperplane 
	\[H_I= \ker\sum_{i\in I}x_i\]
	and $G[I]$ is the induced subgraph on the vertices $I\subseteq N$. 
\end{defn}

By this definition, all defining linear forms of a connected subgraph arrangement 
are $0/1$-vectors. 
The following families of graphs play a 
central role in the investigation of free simple connected subgraph arrangements. They are equally crucial for our study.

\begin{defn}[\cite{cuntzkuehne:subgrapharrangements}]\label{definition: cuntz kühne graph families}
	\begin{itemize}
		\item The \emph{path-graph} $P_n$ on $n$ vertices.
		\begin{figure}[H]
			\begin{tikzpicture}
				\node (t) at (-1,0) {$P_n$:};
				\node (v1) at (0,0) {};
				\node (v2) at (2,0) {};
				\node (v2h) at (2.5,0) {};
				\node (vmid) at (3,0) {};
				\node (vn-1h) at (3.5,0) {};
				\node (vn-1) at (4,0) {};
				\node (vn) at (6,0) {};
				\fill ($(v1)$) circle[radius=2pt];
				\fill ($(v2)$) circle[radius=2pt];
				\fill ($(vn-1)$) circle[radius=2pt];
				\fill ($(vn)$) circle[radius=2pt];
				\node[below] at ($(v1)$) {$1$};
				\node[below] at ($(v2)$) {$2$};
				\node[below] at ($(vn-1)$) {$n-1$};
				\node[below] at ($(vn)$) {$n$};
				\draw ($(v1)$) -- ($(v2)$) -- ($(v2h)$);
				\draw[dashed] ($(v2h)$) -- ($(vn-1h)$);
				\draw ($(vn-1h)$) -- ($(vn-1)$) -- ($(vn)$);
			\end{tikzpicture}
		\end{figure}
		\item The \emph{almost-path-graph} $A_{n,k}$ on $n+1$ vertices, where $1\leq k\leq n$. Draw a path-graph on $n$ vertices, add an additional vertex $n+1$, and connect the vertices $k$ and $n+1$ by an edge.
		
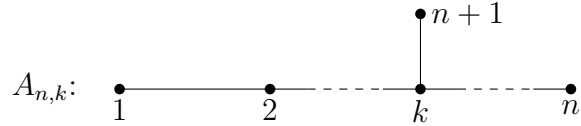
\begin{figure}[H]
			
			\begin{tikzpicture}
				\node (t) at (-1,0) {$A_{n,k}$:};
				\node (v1) at (0,0) {};
				\node (v2) at (2,0) {};
				\node (v2h+) at (2.5,0) {};
				\node (v2kmid) at (3,0) {};
				\node (vkh-) at (3.5,0) {};
				\node (vk) at (4,0) {};
				\node (vn+1) at (4,1) {};
				\node (vkh+) at (4.5,0) {};
				\node (vknmid) at (5,0) {};
				\node (vnh-) at (5.5,0) {};
				\node (vn) at (6,0) {};
				\fill ($(v1)$) circle[radius=2pt];
				\fill ($(v2)$) circle[radius=2pt];
				\fill ($(vk)$) circle[radius=2pt];
				\fill ($(vn)$) circle[radius=2pt];
				\fill ($(vn+1)$) circle[radius=2pt];
				\node[below] at ($(v1)$) {$1$};
				\node[below] at ($(v2)$) {$2$};
				\node[below] at ($(vk)$) {$k$};
				\node[below] at ($(vn)$) {$n$};
				\node[right] at ($(vn+1)$) {$n+1$};
				\draw ($(v1)$) -- ($(v2)$) -- ($(v2h+)$);
				\draw[dashed] ($(v2h+)$) -- ($(vkh-)$);
				\draw ($(vkh-)$) -- ($(vk)$) -- ($(vkh+)$);
				\draw ($(vk)$) -- ($(vn+1)$);
				\draw[dashed] ($(vkh+)$) -- ($(vnh-)$);
				\draw ($(vnh-)$) -- ($(vn)$);
			\end{tikzpicture}
		\end{figure}
		\item The \emph{cycle-graph} $C_n$ on $n$ vertices. Draw $n$ vertices and connect vertex $1$ and $2$, $2$ and $3$, $\dots$, $n$ and $1$.
		
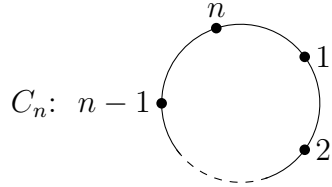
\begin{figure}[H]
			\begin{tikzpicture}
				
				\node (t) at (-2.75,0) {$C_{n}$:};
				\node (v1) at (0.85065080835203988,0.61803398874989468) {};  
				\node (v2) at (0.85065080835203988,-0.61803398874989468) {}; 
				\node (v2+) at (0.32491969623290623,-1.) {}; 
				\node (vn-1-) at (-0.85065080835203988,-0.61803398874989468) {};  
				\node (vn-1) at (-1.0514622242382672,0.0) {}; 
				\node (vn) at (-0.32491969623290623,1.) {};
				\fill ($(v1)$) circle[radius=2pt];
				\fill ($(v2)$) circle[radius=2pt];
				\fill ($(vn-1)$) circle[radius=2pt];
				\fill ($(vn)$) circle[radius=2pt];
				\node[right] at ($(v1)$) {$1$};
				\node[right] at ($(v2)$) {$2$};
				\node[left] at ($(vn-1)$) {$n-1$};
				\node[above] at ($(vn)$) {$n$};
				
				\draw ($(v2+)$) arc (-72:220:1.0514622242382672) ;
				\draw[dashed] ($(vn-1-)$) arc (216:290:1.0514622242382672) ;
			\end{tikzpicture}
		\end{figure}
		\item The \emph{path-with-triangle-graph} $\Delta_{n,k}$, where $1\leq k\leq n-1$. Draw a path-graph on $n$ vertices, add an additional vertex $n+1$, and connect the vertices $k$ and $n+1$ by an edge as well as the vertices $k+1$ and $n+1$.
		
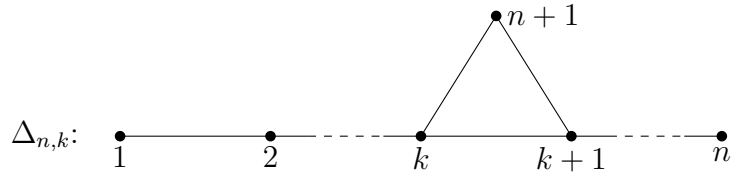
\begin{figure}[H]
			\begin{tikzpicture}
				\node (t) at (-1,0) {$\Delta_{n,k}$:};
				\node (v1) at (0,0) {};
				\node (v2) at (2,0) {};
				\node (v2h+) at (2.5,0) {};
				\node (v2kmid) at (3,0) {};
				\node (vkh-) at (3.5,0) {};
				\node (vk) at (4,0) {};
				\node (vn+1) at (5,1.6) {};
				\node (vk+1) at (6,0) {};
				\node (vk+1h+) at (6.5,0) {};
				\node (vk+1nmid) at (7,0) {};
				\node (vnh-) at (7.5,0) {};
				\node (vn) at (8,0) {};
				\fill ($(v1)$) circle[radius=2pt];
				\fill ($(v2)$) circle[radius=2pt];
				\fill ($(vk)$) circle[radius=2pt];
				\fill ($(vk+1)$) circle[radius=2pt];
				\fill ($(vn)$) circle[radius=2pt];
				\fill ($(vn+1)$) circle[radius=2pt];
				\node[below] at ($(v1)$) {$1$};
				\node[below] at ($(v2)$) {$2$};
				\node[below] at ($(vk)$) {$k$};
				\node[below] at ($(vk+1)$) {$k+1$};
				\node[below] at ($(vn)$) {$n$};
				\node[right] at ($(vn+1)$) {$n+1$};
				\draw ($(v1)$) -- ($(v2)$) -- ($(v2h+)$);
				\draw[dashed] ($(v2h+)$) -- ($(vkh-)$);
				\draw ($(vkh-)$) -- ($(vk)$) -- ($(vk+1)$) -- ($(vk+1h+)$);
				\draw ($(vk)$) -- ($(vn+1)$) -- ($(vk+1)$);
				\draw[dashed] ($(vk+1h+)$) -- ($(vnh-)$);
				\draw ($(vnh-)$) -- ($(vn)$);
			\end{tikzpicture}
		\end{figure}
	\end{itemize}
\end{defn}

In \cite{yoshinaga:extendable}, M.\ Yoshinaga investigated the freeness question of multiplicities for
a certain class of arrangements. His setting turns out to be important for our study as well.

\begin{defn}[{\cite[Def.~2.1]{yoshinaga:extendable}}]
	Let $\CA$ be a hyperplane arrangement.
	\begin{itemize}
		\item[(i)] We call $\CA$ \emph{locally $A_2$}, if $\vert\CA_X\vert\leq 3$ for all $X\in L(\CA)_2$.
		
		\item[(ii)] Fix a reduced system of defining forms $\Phi \subseteq V^*$ for $\CA$, i.e., $\CA = \{\ker(\alpha) \mid \alpha \in \Phi\}$ and
		$|\Phi| = |\CA|$.
		Then $\Phi$ is called a \emph{positive system} for the locally $A_2$ arrangement $\CA$ provided for each rank $2$ localization
		with three elements $\CA_X=\{\ker(\alpha),\ker(\beta),\ker(\gamma)\}$ ($\alpha,\beta,\gamma \in \Phi$) 
		one of $\alpha=\beta+\gamma$, $\beta=\alpha+\gamma$, $\gamma=\alpha+\beta$ holds.
	\end{itemize}
\end{defn}

Connected subgraph arrangements are locally $A_2$ and have positive systems.

\begin{lemma}[{\cite[Lem.~2.10]{cuntzkuehne:subgrapharrangements}}]\label{lemma: connected subgraph arrangements are locallyA2}
	For pairwise distinct nonempty subsets $A_1, A_2, A_3\subseteq  [n]$ and $n\geq1$ the following two conditions are equivalent:
	\begin{itemize}
		\item[(i)] $\rank(H_{A_1}\cap H_{A_2}\cap H_{A_3})=2$ and
		\item[(ii)] $A_{i_1}\dot{\cup}A_{i_2} = A_{i_3}$ for some choice of pairwise distinct indices $1 \leq i_1, i_2, i_3 \leq 3$ where $\dot{\cup}$
		denotes a disjoint set union.
	\end{itemize}
\end{lemma}

A property for arrangements is called \emph{local} if it is preserved under localizations, \cite[Def.~6.1]{cuntzkuehne:subgrapharrangements}.
Typical examples of local properties of arrangements are freeness \cite[Thm.~4.37]{orlikterao:arrangements} and 
inductive freeness \cite[Thm~1.1]{hogeroehrleschauenburg:free}.
We record two graph theoretic constructions from
\cite{cuntzkuehne:subgrapharrangements} which allow us to
control local properties of the $\CA_G$.

\begin{defn}
\label{def:localag}
	Let $G=(N,E)$ be a graph and $e=\{i,j\}\in E$ an edge. We denote by $G/e$ the graph with contracted edge $e$, i.e. the graph in which the vertices $i,j$ are identified to one vertex $k$, each edge to $i$ or $j$ becomes an edge to $k$, and multiple edges and loops are removed.
	For $S \subseteq N$ the \emph{induced subgraph} of $G$ on $S$ we denote by $G[S] = (S,E\cap\binom{S}{2})$.
\end{defn}

We have the following link between these operations and localizations of $\CA_G$.

\begin{lemma}[{\cite[Lem.~6.2, Lem.~6.4]{cuntzkuehne:subgrapharrangements}}]
	\label{lemma:ind_subgraph contract localization}
	Let $G=(N,E)$ be a graph. Then
	\begin{itemize}
		\item[(i)] for $S \subseteq N$ we have $\CA_{G[S]} = (\CA_G)_X$ for some $X \in L(\CA_G)$, and
		\item[(ii)] for $e \in E$, $\CA_{G/{e}} = (\CA_G)_X$ for some $X \in L(\CA_G)$.
	\end{itemize}
	In particular, if $P$ is a local arrangement property satisfied by $\CA_G$, then both $\CA_{G[S]}$ and $\CA_{G/e}$ satisfy $P$ as well.
\end{lemma}

Cuntz and Kühne classified all graphs $G$ such that $\CA_G$ is a free arrangement. It turns out that $\CA_G$ is free if and only if $G$ belongs to one of the four families of graphs introduced in Definition \ref{definition: cuntz kühne graph families}.

\begin{theorem}[{\cite[Thm.~1.6]{cuntzkuehne:subgrapharrangements}}]\label{theorem: FreeConnectedSubgraphArrangements}
	Let $G$ be a connected graph. The connected subgraph arrangement $\CA_G$ is free if and only if $G$ is a path-graph, a cycle-graph, an almost-path-graph, or a path-with-triangle-graph.
\end{theorem}

In their proof, Cuntz and Kühne  show that every graph that is not part of one of the four families from Definition \ref{definition: cuntz kühne graph families}
can get modified using iterations of the constructions from  Definition \ref{def:localag} until we end up with one of the graphs depicted in Figure \ref{fig:G1_8}. Note that $G_2$ is more commonly known as $K_4$.
Since the connected subgraph arrangements stemming from those eight graphs are not free
and freeness is a local property by Theorem \ref{theorem: Multilocalizations are free}, Lemma \ref{lemma:ind_subgraph contract localization} finishes the proof.

\begin{figure}[H]
	\begin{tikzpicture}
		\node (a1) at (2,2) {};
		\node (a2) at (0,2)  {};
		\node (a3) at (0,0)  {};
		\node (a4) at (2,0)  {};
		\foreach \from/\to in {a1/a2,a2/a3,a3/a4,a1/a4,a1/a3}
		\draw ($(\from)$) -- ($(\to)$);
		\node[right] at ($(a1)$) {$1$};
		\node[left] at ($(a2)$) {$2$};
		\node[left] at ($(a3)$) {$3$};
		\node[right] at ($(a4)$) {$4$};
		\fill ($(a1)$) circle[radius=2pt];
		\fill ($(a2)$) circle[radius=2pt];
		\fill ($(a3)$) circle[radius=2pt];
		\fill ($(a4)$) circle[radius=2pt];
		\node at (1,-0.5) {$G_1$};
	\end{tikzpicture}\quad 
	\begin{tikzpicture}
		\node (a1) at (2,2) {};
		\node (a2) at (0,2)  {};
		\node (a3) at (0,0)  {};
		\node (a4) at (2,0)  {};
		\foreach \from/\to in {a1/a2,a2/a3,a3/a4,a1/a4,a1/a3,a2/a4}
		\draw ($(\from)$) -- ($(\to)$);
		\node[right] at ($(a1)$) {$1$};
		\node[left] at ($(a2)$) {$2$};
		\node[left] at ($(a3)$) {$3$};
		\node[right] at ($(a4)$) {$4$};
		\fill ($(a1)$) circle[radius=2pt];
		\fill ($(a2)$) circle[radius=2pt];
		\fill ($(a3)$) circle[radius=2pt];
		\fill ($(a4)$) circle[radius=2pt];
		\node at (1,-0.5) {$G_2$};
	\end{tikzpicture}\quad
	\begin{tikzpicture}
		\node (a1) at (0,0) {};
		\node (a2) at (1,1)  {};
		\node (a3) at (-1,1)  {};
		\node (a4) at (-1,-1)  {};
		\node (a5) at (1,-1)  {};
		\foreach \from/\to in {a1/a2,a1/a3,a1/a4,a1/a5}
		\draw ($(\from)$) -- ($(\to)$);
		\node[below] at ($(a1)$) {$1$};
		\node[right] at ($(a2)$) {$2$};
		\node[left] at ($(a3)$) {$3$};
		\node[left] at ($(a4)$) {$4$};
		\node[right] at ($(a5)$) {$5$};
		\fill ($(a1)$) circle[radius=2pt];
		\fill ($(a2)$) circle[radius=2pt];
		\fill ($(a3)$) circle[radius=2pt];
		\fill ($(a4)$) circle[radius=2pt];
		\fill ($(a5)$) circle[radius=2pt];
		\node at (0,-1.5) {$G_3$};
	\end{tikzpicture}\quad
	\begin{tikzpicture}
		\node (a1) at (0,0) {};
		\node (a2) at (1,1)  {};
		\node (a3) at (-1,1)  {};
		\node (a4) at (-1,-1)  {};
		\node (a5) at (1,-1)  {};
		\foreach \from/\to in {a1/a2,a1/a3,a1/a4,a1/a5,a2/a3}
		\draw ($(\from)$) -- ($(\to)$);
		\node[below] at ($(a1)$) {$1$};
		\node[right] at ($(a2)$) {$2$};
		\node[left] at ($(a3)$) {$3$};
		\node[left] at ($(a4)$) {$4$};
		\node[right] at ($(a5)$) {$5$};
		\fill ($(a1)$) circle[radius=2pt];
		\fill ($(a2)$) circle[radius=2pt];
		\fill ($(a3)$) circle[radius=2pt];
		\fill ($(a4)$) circle[radius=2pt];
		\fill ($(a5)$) circle[radius=2pt];
		\node at (0,-1.5) {$G_4$};
	\end{tikzpicture}
	
	\begin{tikzpicture}
		\node (a1) at (0,0) {};
		\node (a2) at (1,1)  {};
		\node (a3) at (-1,1)  {};
		\node (a4) at (-1,-1)  {};
		\node (a5) at (1,-1)  {};
		\foreach \from/\to in {a1/a2,a1/a3,a1/a4,a1/a5,a2/a3,a4/a5}
		\draw ($(\from)$) -- ($(\to)$);
		\node[below] at ($(a1)$) {$1$};
		\node[right] at ($(a2)$) {$2$};
		\node[left] at ($(a3)$) {$3$};
		\node[left] at ($(a4)$) {$4$};
		\node[right] at ($(a5)$) {$5$};
		\fill ($(a1)$) circle[radius=2pt];
		\fill ($(a2)$) circle[radius=2pt];
		\fill ($(a3)$) circle[radius=2pt];
		\fill ($(a4)$) circle[radius=2pt];
		\fill ($(a5)$) circle[radius=2pt];
		\node at (0,-1.5) {$G_5$};
	\end{tikzpicture}\quad
	\begin{tikzpicture}
		\node (a5) at (0,0) {};
		\node (a1) at (1,1)  {};
		\node (a2) at (-1,1)  {};
		\node (a3) at (-1,-1)  {};
		\node (a4) at (1,-1)  {};
		\foreach \from/\to in {a1/a2,a2/a3,a3/a4,a1/a4,a1/a5}
		\draw ($(\from)$) -- ($(\to)$);
		\node[right] at ($(a1)$) {$1$};
		\node[left] at ($(a2)$) {$2$};
		\node[left] at ($(a3)$) {$3$};
		\node[right] at ($(a4)$) {$4$};
		\node[below] at ($(a5)$) {$5$};
		\fill ($(a1)$) circle[radius=2pt];
		\fill ($(a2)$) circle[radius=2pt];
		\fill ($(a3)$) circle[radius=2pt];
		\fill ($(a4)$) circle[radius=2pt];
		\fill ($(a5)$) circle[radius=2pt];
		\node at (0,-1.5) {$G_6$};
	\end{tikzpicture}\quad
	\begin{tikzpicture}
		\node (a1) at (0,1) {};
		\node (a2) at (-0.6,0) {};
		\node (a3) at (0.6,0) {};
		\node (a4) at (0,1.6) {};
		\node (a5) at (-1.2,-0.5) {};
		\node (a6) at (1.2,-0.5) {};
		\foreach \from/\to in {a1/a2,a1/a3,a2/a3,a1/a4,a2/a5,a3/a6}
		\draw ($(\from)$) -- ($(\to)$);
		\node[right] at ($(a1)$) {$1$};
		\node[below] at ($(a2)$) {$2$};
		\node[below] at ($(a3)$) {$3$};
		\node[right] at ($(a4)$) {$4$};
		\node[left] at ($(a5)$) {$5$};
		\node[right] at ($(a6)$) {$6$};
		\fill ($(a1)$) circle[radius=2pt];
		\fill ($(a2)$) circle[radius=2pt];
		\fill ($(a3)$) circle[radius=2pt];
		\fill ($(a4)$) circle[radius=2pt];
		\fill ($(a5)$) circle[radius=2pt];
		\fill ($(a6)$) circle[radius=2pt];
		\node at (0,-1) {$G_7$};
	\end{tikzpicture}\quad
	\begin{tikzpicture}
	\node (a1) at (0,0.5) {};
	\node (a2) at (0,1) {};
	\node (a3) at (-0.6,0) {};
	\node (a4) at (0.6,0) {};
	\node (a5) at (0,1.6) {};
	\node (a6) at (-1.2,-0.5) {};
	\node (a7) at (1.2,-0.5) {};
	\foreach \from/\to in {a1/a2,a1/a3,a1/a4,a2/a5,a3/a6,a4/a7}
	\draw ($(\from)$) -- ($(\to)$);
	\node[below] at ($(a1)$) {$1$};
	\node[right] at ($(a2)$) {$4$};
	\node[below] at ($(a3)$) {$2$};
	\node[below] at ($(a4)$) {$6$};
	\node[right] at ($(a5)$) {$5$};
	\node[left] at ($(a6)$) {$3$};
	\node[right] at ($(a7)$) {$7$};
	\fill ($(a1)$) circle[radius=2pt];
	\fill ($(a2)$) circle[radius=2pt];
	\fill ($(a3)$) circle[radius=2pt];
	\fill ($(a4)$) circle[radius=2pt];
	\fill ($(a5)$) circle[radius=2pt];
	\fill ($(a6)$) circle[radius=2pt];
	\fill ($(a7)$) circle[radius=2pt];
	\node at (0,-1) {$G_8$};
	\end{tikzpicture}
	\caption{The graphs $G_1$ up to $G_8$.}
	\label{fig:G1_8}
\end{figure}
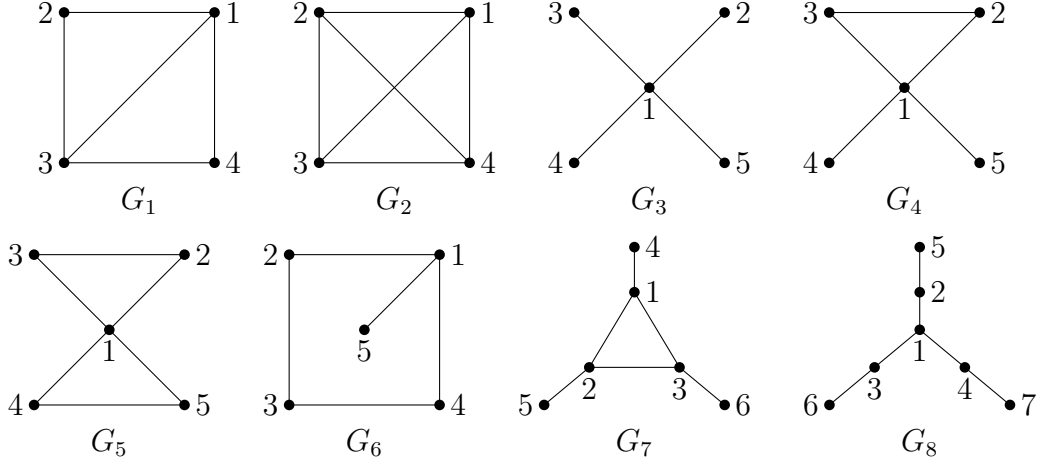

Our aim in this note is to look at these classes within the framework of multiarrangements and to extend this classification. 

\subsection{Extensions of a multiarrangement}
The following definitions and results are due to Yoshinaga \cite{yoshinaga:extendable}. 

\begin{defn}\label{definition: Yoshinga extension}
	Let $(\CA,\mu)$ be a free multiarrangement in $\mathbb{C}^\ell$. 
	\begin{itemize}
		\item[(i)] $(\CA,\mu)$ is \emph{extendable} if it can be obtained as a Ziegler restriction of a free arrangement in $\mathbb{C}^{\ell+1}$.
		\item[(ii)] Define an extension in $\mathbb{C}^{\ell+1}$ 
		with $(x_1,x_2,\dots,x_\ell,x_{\ell+1})$ as a coordinate system by
		\[\CE(\CA,\mu)=\{\ker(x_{\ell+1})\}\cup\left\{\ker(\alpha_H-kx_{\ell+1})\ \middle| \  k\in\mathbb{Z},-\frac{\mu(H)-1}{2}\leq k\leq \frac{\mu(H)}{2}, H \in \CA\right\}.\]
		Note that $(\CE(\CA,\mu)^{\ker(x_{\ell+1})},\kappa)=(\CA,\mu)$.
	\end{itemize}
\end{defn}

One of the main results in \cite{yoshinaga:extendable} is the following.

\begin{theorem}[{\cite[Thm.~2.5]{yoshinaga:extendable}}]\label{theorem: Yoshinaga extendable theorem}
	Let $\CA$ be a locally $A_2$ arrangement with a positive system in $V=\mathbb{C}^\ell$. We fix a positive system $\Phi^+=\{\alpha_H\mid H\in\CA\}\subset V^*$ of defining equations. Let $\mu:\CA\to \mathbb{Z}_{\geq 0}$ be a multiplicity. Assume the following condition:
\begin{equation}
\label{eq:star} 
\parbox[c]{14cm}{\text{Let $\CA_X=\{H_i,H_j,H_k\}$ be a codimension two localization with $\alpha_i=\alpha_j+\alpha_k$.} 
\text{If $\mu(H_i)$ is odd, then at least one of $\mu(H_j)$ or $\mu(H_k)$ is odd.}
}
\end{equation}

	Then $(\CA,\mu)$ is free, if and only if it is extendable. Indeed, $\CE(\CA,\mu)$ is a free arrangement.
\end{theorem}

Note that a free multiarrangement $(\CA,\mu)$ need not be extendable in general, 
\cite[Ex.~2.8]{yoshinaga:extendable}.

\subsection{Locally and globally mixed products}
In \cite{abeteraowakefield:euler}, 
a very useful tool is introduced 
to determine whether a given multiarrangement $(\CA,\mu)$ fails to be free.

\begin{defn}
\label{def:lmpgmp}
	Let $(\CA,\mu)$ be a multiarrangement in $\mathbb{C}^\ell$. Let $0 \le k \le \ell$. 
	\begin{itemize}
		\item[(i)] If the localization $(\CA_X,\mu_X)$ is free with exponents $\exp(\CA_X,\mu_X)=(e_1^X,e_2^X,\dots, e_k^X)$ for each $X\in L(\CA)_k$, then define the \emph{$k$-th local mixed product} by
		\[\LMP (k)=\sum_{X\in L_k} e_1^X e_2^X\cdots e_k^X.\] 
		\item[(ii)] If $(\CA,\mu)$ is free with exponents $\exp(\CA,\mu)=(e_1,\dots,e_\ell)$, define the \emph{$k$-th global mixed product} by \[\GMP (k)=\sum e_{i_1}e_{i_2}\cdots e_{i_k},\] where the sum is taken over all $k$-tuples from $\exp(\CA,\mu)$ with $1\leq i_1<\dots<i_k\leq \ell$.
	\end{itemize}
\end{defn}

Thanks to Remark \ref{rem:rank2indfree}, $\LMP (2)$ is always defined. These invariants are rather useful for proving the non-freeness of a given multiarrangement due to the following result.

\begin{theorem}[{\cite[Cor.~4.6]{abeteraowakefield:multichapol}}]
	\label{MixedEquality}
	If $(\CA,\mu)$ is free with $\exp(\CA,\mu)=(e_1,\dots, e_\ell)$ then for every $0\leq k\leq \ell$, we have 
	\[\GMP (k)=\LMP (k).\]
\end{theorem}

\begin{defn}
\label{def:lmp}
For $1 \le i \le \ell$, let $e_i \in \mathbb{N}$. We use the shorthand $(e_1, \ldots , e_\ell)_\le \in \mathbb{N}^\ell$ for $e_1 \le e_2 \le \ldots \le e_\ell$. 
Then for the latter define 
the \emph{$k$-th global mixed product} by \[\GMP (k) :=\sum e_{i_1}e_{i_2}\cdots e_{i_k},\]
where the sum is taken over all $k$-tuples from $(e_1, \ldots , e_\ell)_\le$ with $1\leq i_1<\dots<i_k\leq \ell$.
\end{defn}

\begin{remark}
	\label{rem:GMP_LMP_balanced}
Following \cite{abeteraowakefield:multichapol}, given two ordered $\ell$-tuples of non-negative integers, $(d_1,d_2,\dots,d_\ell)_\le$ and $(e_1,e_2,\dots,e_\ell)_\le$ as above satisfying $\sum d_i=\sum e_i$, we call $(d_1,d_2,\dots,d_\ell)_\le$ \emph{more balanced} than $(e_1,e_2,\dots,e_\ell)_\le$ if $d_\ell-d_1<e_\ell-e_1$. 
It follows from \cite[p.836]{abeteraowakefield:multichapol}
that if $(d_1,d_2,\dots,d_\ell)_\le$ is more balanced than $(e_1,e_2,\dots,e_\ell)_\le$, then $\GMP (k)$ is larger for $(d_1,d_2,\dots,d_\ell)_\le$ than for $(e_1,e_2,\dots,e_\ell)_\le$.
\end{remark}

\begin{example}
    With notation as in Definition \ref{def:lmp} let $S_1=(4,4,4,4)_{\leq}$ and $S_2=(2,4,5,5)_{\leq}$. Then $S_1$ is more balanced than $S_2$ since $4-4 < 5-2$. Among all quadruples $(d_1,d_2,d_3,d_4)_{\leq}$ with $\sum d_i=16$, we see that $S_1$ is most balanced.    
\end{example}

The following is immediate from Theorem \ref{MixedEquality} and
Remark \ref{rem:GMP_LMP_balanced}. With respect to part (ii) of the following lemma, recall that  $1\in\exp(\CA)$ if $\CA$ is free.

\begin{lemma}
	\label{lemma:LMP_GMP_nonfree}
 Let $\CA$ be a simple arrangement in $\mathbb{K}^\ell$. 
\begin{itemize}
    \item [(i)]  Let $\mu$ be a multiplicity on $\CA$ and let $(e_1,e_2,\dots,e_\ell)_\leq$ be the unique most balanced ordered $\ell$-tuple 
	of integers with sum equal to $\vert\mu\vert$. If the $\GMP(2)$ for $(e_1,\ldots,e_\ell)_\leq$ is smaller than the $\LMP (2)$ for $(\CA,\mu)$, then $(\CA,\mu)$ is not free.
    \item [(ii)]  Let
	$(e_2,\dots,e_\ell)_\leq$ be the unique most balanced ordered $\ell$-tuple
	of integers with sum equal to $\vert\CA\vert-1$. If the $\GMP(2)$ for $(1,e_2,\ldots,e_\ell)_\leq$ is smaller than the $\LMP (2)$ for $\CA$, then $\CA$ is not free.
\end{itemize}
    \end{lemma}

\section{Free connected subgraph multiarrangements}
\label{sect:FreeConnSubgraphMulti}
The goal of this section is a generalization of 
Theorem \ref{theorem: FreeConnectedSubgraphArrangements}, where we give a list of all graphs $G$ for which there exists a multiplicity $\mu:\CA\to\mathbb{Z}_{>0}$
such that $(\CA,\mu)$ is free. 
As mentioned before, Cuntz and Kühne derive Theorem \ref{theorem: FreeConnectedSubgraphArrangements} by showing that any graph not listed in 
Definition \ref{definition: cuntz kühne graph families} can get manipulated by graph contractions from Definition \ref{def:localag} such that the resulting graph possesses an induced subgraph that affords a non-free arrangement. 
Unexpectedly, it turns out that free multiplicities do exist for $\CA_{G_1}$ and $\CA_{G_2}$. 
Thus, in order to follow a similar line of argument, as the one above, we need to expand the list of graphs 
we have to consider which are shown in Figure \ref{fig:G9_20}.

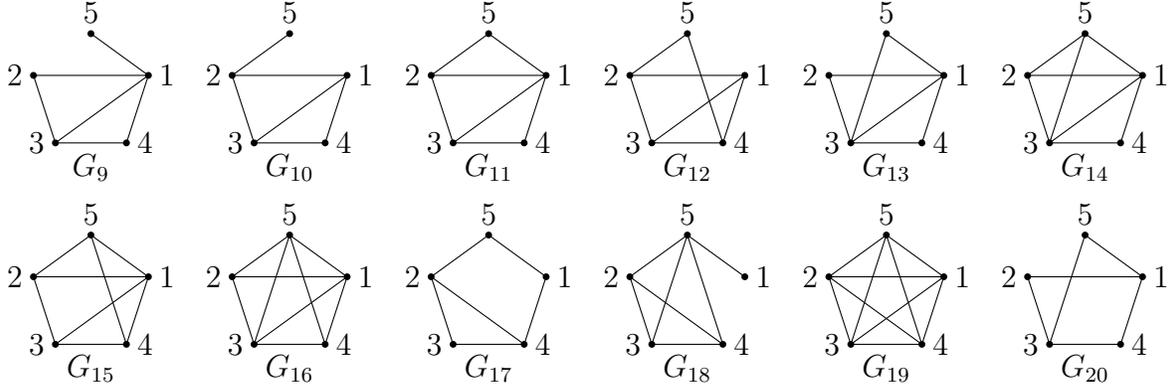
\begin{figure}[h]
	\def \sc {0.65}
	\begin{tikzpicture}[scale=\sc]		
		\node (a1) at (1.1755705045849463,0.38196601125010554) {};
		\node (a2) at (-1.1755705045849463,0.38196601125010554) {};
		\node (a3) at (-0.72654252800536101,-1) {};
		\node (a4) at (0.72654252800536101,-1) {};
		\node (a5) at (0.0,1.2360679774997894) {};
		\foreach \from/\to in {a1/a2,a2/a3,a3/a4,a1/a4,a1/a3,a1/a5}
		\draw ($(\from)$) -- ($(\to)$);
		\node[right] at ($(a1)$) {$1$};
		\node[left] at ($(a2)$) {$2$};
		\node[left] at ($(a3)$) {$3$};
		\node[right] at ($(a4)$) {$4$};
		\node[above] at ($(a5)$) {$5$};
		\fill ($(a1)$) circle[radius=2pt];
		\fill ($(a2)$) circle[radius=2pt];
		\fill ($(a3)$) circle[radius=2pt];
		\fill ($(a4)$) circle[radius=2pt];
		\fill ($(a5)$) circle[radius=2pt];
		\node at (0,-1.5) {$G_9$};
	\end{tikzpicture}
	\begin{tikzpicture}[scale=\sc]		
		\node (a1) at (1.1755705045849463,0.38196601125010554) {};
		\node (a2) at (-1.1755705045849463,0.38196601125010554) {};
		\node (a3) at (-0.72654252800536101,-1) {};
		\node (a4) at (0.72654252800536101,-1) {};
		\node (a5) at (0.0,1.2360679774997894) {};
		\foreach \from/\to in {a1/a2,a2/a3,a3/a4,a1/a4,a1/a3,a2/a5}
		\draw ($(\from)$) -- ($(\to)$);
		\node[right] at ($(a1)$) {$1$};
		\node[left] at ($(a2)$) {$2$};
		\node[left] at ($(a3)$) {$3$};
		\node[right] at ($(a4)$) {$4$};
		\node[above] at ($(a5)$) {$5$};
		\fill ($(a1)$) circle[radius=2pt];
		\fill ($(a2)$) circle[radius=2pt];
		\fill ($(a3)$) circle[radius=2pt];
		\fill ($(a4)$) circle[radius=2pt];
		\fill ($(a5)$) circle[radius=2pt];
		\node at (0,-1.5) {$G_{10}$};
	\end{tikzpicture}
	\begin{tikzpicture}[scale=\sc]		
		\node (a1) at (1.1755705045849463,0.38196601125010554) {};
		\node (a2) at (-1.1755705045849463,0.38196601125010554) {};
		\node (a3) at (-0.72654252800536101,-1) {};
		\node (a4) at (0.72654252800536101,-1) {};
		\node (a5) at (0.0,1.2360679774997894) {};
		\foreach \from/\to in {a1/a2,a2/a3,a3/a4,a1/a4,a1/a3,a2/a5,a1/a5}
		\draw ($(\from)$) -- ($(\to)$);
		\node[right] at ($(a1)$) {$1$};
		\node[left] at ($(a2)$) {$2$};
		\node[left] at ($(a3)$) {$3$};
		\node[right] at ($(a4)$) {$4$};
		\node[above] at ($(a5)$) {$5$};
		\fill ($(a1)$) circle[radius=2pt];
		\fill ($(a2)$) circle[radius=2pt];
		\fill ($(a3)$) circle[radius=2pt];
		\fill ($(a4)$) circle[radius=2pt];
		\fill ($(a5)$) circle[radius=2pt];
		\node at (0,-1.5) {$G_{11}$};
	\end{tikzpicture}
	\begin{tikzpicture}[scale=\sc]		
		\node (a1) at (1.1755705045849463,0.38196601125010554) {};
		\node (a2) at (-1.1755705045849463,0.38196601125010554) {};
		\node (a3) at (-0.72654252800536101,-1) {};
		\node (a4) at (0.72654252800536101,-1) {};
		\node (a5) at (0.0,1.2360679774997894) {};
		\foreach \from/\to in {a1/a2,a2/a3,a3/a4,a1/a4,a1/a3,a2/a5,a4/a5}
		\draw ($(\from)$) -- ($(\to)$);
		\node[right] at ($(a1)$) {$1$};
		\node[left] at ($(a2)$) {$2$};
		\node[left] at ($(a3)$) {$3$};
		\node[right] at ($(a4)$) {$4$};
		\node[above] at ($(a5)$) {$5$};
		\fill ($(a1)$) circle[radius=2pt];
		\fill ($(a2)$) circle[radius=2pt];
		\fill ($(a3)$) circle[radius=2pt];
		\fill ($(a4)$) circle[radius=2pt];
		\fill ($(a5)$) circle[radius=2pt];
		\node at (0,-1.5) {$G_{12}$};
	\end{tikzpicture}
	\begin{tikzpicture}[scale=\sc]		
		\node (a1) at (1.1755705045849463,0.38196601125010554) {};
		\node (a2) at (-1.1755705045849463,0.38196601125010554) {};
		\node (a3) at (-0.72654252800536101,-1) {};
		\node (a4) at (0.72654252800536101,-1) {};
		\node (a5) at (0.0,1.2360679774997894) {};
		\foreach \from/\to in {a1/a2,a2/a3,a3/a4,a1/a4,a1/a3,a1/a5,a3/a5}
		\draw ($(\from)$) -- ($(\to)$);
		\node[right] at ($(a1)$) {$1$};
		\node[left] at ($(a2)$) {$2$};
		\node[left] at ($(a3)$) {$3$};
		\node[right] at ($(a4)$) {$4$};
		\node[above] at ($(a5)$) {$5$};
		\fill ($(a1)$) circle[radius=2pt];
		\fill ($(a2)$) circle[radius=2pt];
		\fill ($(a3)$) circle[radius=2pt];
		\fill ($(a4)$) circle[radius=2pt];
		\fill ($(a5)$) circle[radius=2pt];
		\node at (0,-1.5) {$G_{13}$};
	\end{tikzpicture}
	\begin{tikzpicture}[scale=\sc]		
		\node (a1) at (1.1755705045849463,0.38196601125010554) {};
		\node (a2) at (-1.1755705045849463,0.38196601125010554) {};
		\node (a3) at (-0.72654252800536101,-1) {};
		\node (a4) at (0.72654252800536101,-1) {};
		\node (a5) at (0.0,1.2360679774997894) {};
		\foreach \from/\to in {a1/a2,a2/a3,a3/a4,a1/a4,a1/a3,a2/a5,a1/a5,a3/a5}
		\draw ($(\from)$) -- ($(\to)$);
		\node[right] at ($(a1)$) {$1$};
		\node[left] at ($(a2)$) {$2$};
		\node[left] at ($(a3)$) {$3$};
		\node[right] at ($(a4)$) {$4$};
		\node[above] at ($(a5)$) {$5$};
		\fill ($(a1)$) circle[radius=2pt];
		\fill ($(a2)$) circle[radius=2pt];
		\fill ($(a3)$) circle[radius=2pt];
		\fill ($(a4)$) circle[radius=2pt];
		\fill ($(a5)$) circle[radius=2pt];
		\node at (0,-1.5) {$G_{14}$};
	\end{tikzpicture}
	\begin{tikzpicture}[scale=\sc]		
		\node (a1) at (1.1755705045849463,0.38196601125010554) {};
		\node (a2) at (-1.1755705045849463,0.38196601125010554) {};
		\node (a3) at (-0.72654252800536101,-1) {};
		\node (a4) at (0.72654252800536101,-1) {};
		\node (a5) at (0.0,1.2360679774997894) {};
		\foreach \from/\to in {a1/a2,a2/a3,a3/a4,a1/a4,a1/a3,a2/a5,a1/a5,a4/a5}
		\draw ($(\from)$) -- ($(\to)$);
		\node[right] at ($(a1)$) {$1$};
		\node[left] at ($(a2)$) {$2$};
		\node[left] at ($(a3)$) {$3$};
		\node[right] at ($(a4)$) {$4$};
		\node[above] at ($(a5)$) {$5$};
		\fill ($(a1)$) circle[radius=2pt];
		\fill ($(a2)$) circle[radius=2pt];
		\fill ($(a3)$) circle[radius=2pt];
		\fill ($(a4)$) circle[radius=2pt];
		\fill ($(a5)$) circle[radius=2pt];
		\node at (0,-1.5) {$G_{15}$};
	\end{tikzpicture}
	\begin{tikzpicture}[scale=\sc]		
		\node (a1) at (1.1755705045849463,0.38196601125010554) {};
		\node (a2) at (-1.1755705045849463,0.38196601125010554) {};
		\node (a3) at (-0.72654252800536101,-1) {};
		\node (a4) at (0.72654252800536101,-1) {};
		\node (a5) at (0.0,1.2360679774997894) {};
		\foreach \from/\to in {a1/a2,a2/a3,a3/a4,a1/a4,a1/a3,a2/a5,a1/a5,a3/a5,a4/a5}
		\draw ($(\from)$) -- ($(\to)$);
		\node[right] at ($(a1)$) {$1$};
		\node[left] at ($(a2)$) {$2$};
		\node[left] at ($(a3)$) {$3$};
		\node[right] at ($(a4)$) {$4$};
		\node[above] at ($(a5)$) {$5$};
		\fill ($(a1)$) circle[radius=2pt];
		\fill ($(a2)$) circle[radius=2pt];
		\fill ($(a3)$) circle[radius=2pt];
		\fill ($(a4)$) circle[radius=2pt];
		\fill ($(a5)$) circle[radius=2pt];
		\node at (0,-1.5) {$G_{16}$};
	\end{tikzpicture}
	\begin{tikzpicture}[scale=\sc]		
		\node (a1) at (1.1755705045849463,0.38196601125010554) {};
		\node (a2) at (-1.1755705045849463,0.38196601125010554) {};
		\node (a3) at (-0.72654252800536101,-1) {};
		\node (a4) at (0.72654252800536101,-1) {};
		\node (a5) at (0.0,1.2360679774997894) {};
		\foreach \from/\to in {a1/a4,a2/a3,a3/a4,a2/a5,a2/a4,a1/a5}
		\draw ($(\from)$) -- ($(\to)$);
		\node[right] at ($(a1)$) {$1$};
		\node[left] at ($(a2)$) {$2$};
		\node[left] at ($(a3)$) {$3$};
		\node[right] at ($(a4)$) {$4$};
		\node[above] at ($(a5)$) {$5$};
		\fill ($(a1)$) circle[radius=2pt];
		\fill ($(a2)$) circle[radius=2pt];
		\fill ($(a3)$) circle[radius=2pt];
		\fill ($(a4)$) circle[radius=2pt];
		\fill ($(a5)$) circle[radius=2pt];
		\node at (0,-1.5) {$G_{17}$};
	\end{tikzpicture}
	\begin{tikzpicture}[scale=\sc]		
		\node (a1) at (1.1755705045849463,0.38196601125010554) {};
		\node (a2) at (-1.1755705045849463,0.38196601125010554) {};
		\node (a3) at (-0.72654252800536101,-1) {};
		\node (a4) at (0.72654252800536101,-1) {};
		\node (a5) at (0.0,1.2360679774997894) {};
		\foreach \from/\to in {a2/a3,a3/a4,a2/a4,a3/a5,a2/a5,a1/a5,a4/a5}
		\draw ($(\from)$) -- ($(\to)$);
		\node[right] at ($(a1)$) {$1$};
		\node[left] at ($(a2)$) {$2$};
		\node[left] at ($(a3)$) {$3$};
		\node[right] at ($(a4)$) {$4$};
		\node[above] at ($(a5)$) {$5$};
		\fill ($(a1)$) circle[radius=2pt];
		\fill ($(a2)$) circle[radius=2pt];
		\fill ($(a3)$) circle[radius=2pt];
		\fill ($(a4)$) circle[radius=2pt];
		\fill ($(a5)$) circle[radius=2pt];
		\node at (0,-1.5) {$G_{18}$};
	\end{tikzpicture}
	\begin{tikzpicture}[scale=\sc]		
		\node (a1) at (1.1755705045849463,0.38196601125010554) {};
		\node (a2) at (-1.1755705045849463,0.38196601125010554) {};
		\node (a3) at (-0.72654252800536101,-1) {};
		\node (a4) at (0.72654252800536101,-1) {};
		\node (a5) at (0.0,1.2360679774997894) {};
		\foreach \from/\to in {a1/a2,a1/a3,a1/a4,a1/a5,a2/a3,a2/a4,a2/a5,a3/a4,a3/a5,a4/a5}
		\draw ($(\from)$) -- ($(\to)$);
		\node[right] at ($(a1)$) {$1$};
		\node[left] at ($(a2)$) {$2$};
		\node[left] at ($(a3)$) {$3$};
		\node[right] at ($(a4)$) {$4$};
		\node[above] at ($(a5)$) {$5$};
		\fill ($(a1)$) circle[radius=2pt];
		\fill ($(a2)$) circle[radius=2pt];
		\fill ($(a3)$) circle[radius=2pt];
		\fill ($(a4)$) circle[radius=2pt];
		\fill ($(a5)$) circle[radius=2pt];
		\node at (0,-1.5) {$G_{19}$};
	\end{tikzpicture}
	\begin{tikzpicture}[scale=\sc]		
		\node (a1) at (1.1755705045849463,0.38196601125010554) {};
		\node (a2) at (-1.1755705045849463,0.38196601125010554) {};
		\node (a3) at (-0.72654252800536101,-1) {};
		\node (a4) at (0.72654252800536101,-1) {};
		\node (a5) at (0.0,1.2360679774997894) {};
		\foreach \from/\to in {a1/a2,a1/a4,a1/a5,a2/a3,a3/a4,a3/a5}
		\draw ($(\from)$) -- ($(\to)$);
		\node[right] at ($(a1)$) {$1$};
		\node[left] at ($(a2)$) {$2$};
		\node[left] at ($(a3)$) {$3$};
		\node[right] at ($(a4)$) {$4$};
		\node[above] at ($(a5)$) {$5$};
		\fill ($(a1)$) circle[radius=2pt];
		\fill ($(a2)$) circle[radius=2pt];
		\fill ($(a3)$) circle[radius=2pt];
		\fill ($(a4)$) circle[radius=2pt];
		\fill ($(a5)$) circle[radius=2pt];
		\node at (0,-1.5) {$G_{20}$};
	\end{tikzpicture}
	\caption{Graphs $G_9$ to $G_{20}$.}\label{fig:G9_20}
\end{figure}

We begin with the following observation.
		
\begin{proposition}\label{proposition: G3 to G19 totally non free}
	Let $G$ be one of the graphs $G_3$ to $G_{20}$ listed in Figures \ref{fig:G1_8} and \ref{fig:G9_20}. Then the corresponding connected subgraph arrangement $\CA_G$ is totally non-free.
\end{proposition}

\begin{proof}
	For each $G$ as in the statement above, we present an element $X\in L(\CA_G)$ in Table \ref{table: G3 to G20 totally non-free}  
	such that the corresponding localization $(\CA_G)_X$ is a generic arrangement with $\rank (\CA_G)_X<\vert(\CA_G)_X\vert$. 
	Consequently, owing to Theorems \ref{theorem: Multilocalizations are free} and \ref{theorem: Yoshinaga generic not free} we conclude that $\CA_G$ is totally non-free. 
\end{proof}

	\begin{table}[h]
	\centering
		\noindent\begin{tabular}{cccc}
			$G$ & $X$ & $\rank X$ & $(\CA_{G})_X$ \\
			\hline
			$G_3$ & $H_1\cap H_{123}\cap H_{145}$ & $3$ & $\{H_{1},H_{123},H_{145},H_{12345}\}$ \\
			$G_4$ & $H_1\cap H_{124}\cap H_{135}$ & $3$ & $\{H_{1},H_{124},H_{135},H_{12345}\}$ \\
			$G_5$ & $H_1\cap H_{124}\cap H_{135}$ & $3$ & $\{H_{1},H_{124},H_{135},H_{12345}\}$ \\
			$G_6$ & $H_{12}\cap H_{14}\cap H_{1235}$ & $3$ & $\{H_{12},H_{14},H_{1235},H_{1345}\}$ \\
			$G_7$ & $H_{23}\cap H_{136}\cap H_{1235}\cap H_{1245}$ & $4$ & $\{H_{23},H_{136},H_{1235},H_{1245},H_{12346}\}$ \\
			$G_8$ & $H_{1}\cap H_{146}\cap H_{1245}\cap H_{12367}\cap H_{12467}$ & $5$ & $\{H_{1},H_{146},H_{1245},H_{12367},H_{12467},H_{123456}\}$ \\
			$G_9$ & $H_{1}\cap H_{124}\cap H_{135}$ & $3$ & $\{H_{1},H_{124},H_{135},H_{12345}\}$ \\
			$G_{10}$ & $H_{12}\cap H_{23}\cap H_{1245}$ & $3$ & $\{H_{12},H_{23},H_{1245},H_{2345}\}$ \\
			$G_{11}$ & $H_{13}\cap H_{23}\cap H_{1345}$ & $3$ & $\{H_{13},H_{23},H_{1345},H_{2345}\}$ \\
			$G_{12}$ & $H_{23}\cap H_{34}\cap H_{1235}$ & $3$ & $\{H_{23},H_{34},H_{1235},H_{1345}\}$ \\
			$G_{13}$ & $H_{15}\cap H_{35}\cap H_{1245}$ & $3$ & $\{H_{15},H_{35},H_{1245},H_{2345}\}$ \\
			$G_{14}$ & $H_{15}\cap H_{35}\cap H_{1245}$ & $3$ & $\{H_{15},H_{35},H_{1245},H_{2345}\}$ \\
			$G_{15}$ & $H_{125}\cap H_{135}\cap H_{245}$ & $3$ & $\{H_{125},H_{135},H_{245},H_{345}\}$ \\
			$G_{16}$ & $H_{15}\cap H_{35}\cap H_{1245}$ & $3$ & $\{H_{15},H_{35},H_{1245},H_{2345}\}$ \\
			$G_{17}$ & $H_{24}\cap H_{25}\cap H_{1234}$ & $3$ & $\{H_{24},H_{25},H_{1234},H_{1235}\}$ \\
			$G_{18}$ & $H_{13}\cap H_{14}\cap H_{1235}$ & $3$ & $\{H_{13},H_{14},H_{1235},H_{1245}\}$ \\
			$G_{19}$ & $H_{125}\cap H_{135}\cap H_{245}$ & $3$ & $\{H_{125},H_{135},H_{245},H_{345}\}$ \\
			$G_{20}$ & $H_{15}\cap H_{35}\cap H_{1245}$ & $3$ & $\{H_{15},H_{35},H_{1245},H_{2345}\}$ \\
			\hline\\
		\end{tabular}\caption{Elements $X\in L(\CA_{G_i})$ such that $(\CA_{G_i})_X$ is generic.}\label{table: G3 to G20 totally non-free} 
	\end{table}
			
\begin{corollary}\label{corollary: G3 subgraph not free}
	Let $G$ be a graph that has $G_3$ as a subgraph. Then $\CA_G$ is totally non-free.
\end{corollary}

\begin{proof}
	Let $v_1,v_2,v_3,v_4,v_5$ be the vertices of $G$ included in the subgraph forming $G_3$. The induced subgraph of $G$ on $\{v_1,v_2,v_3,v_4,v_5\}$ 
	is isomorphic to $G_3,$ $G_4,$ $G_5,$ $G_{9},$ $G_{11},$ $G_{13},$ $G_{14},$ $G_{15},$ $G_{16},$ $G_{18}$ or $G_{19}$. 
	Thus, by Proposition \ref{proposition: G3 to G19 totally non free}, the arrangement of this induced subgraph is totally non-free. 
	It then follows from Lemma \ref{lemma:ind_subgraph contract localization} and Theorem \ref{theorem: Multilocalizations are free}
	that $\CA_G$ is totally non-free as well.    
\end{proof}

We are now able to address the main result of this section.

\begin{theorem}\label{theorem: free connected subgrapharrangement rank bigger 4}
	Let $G$ be a connected graph on at least $5$ vertices. Then there exists a multiplicity $\mu \ge \one$ such that the connected subgraph multiarrangement $(\CA_G,\mu)$ is free if and only if $G$ is a path-graph, a cycle-graph, an almost-path-graph, or a path-with-triangle-graph.
		In particular, all other connected graphs with at least $5$ vertices yield totally non-free connected subgraph arrangements.
\end{theorem}

\begin{proof}
	By Theorem \ref{theorem: FreeConnectedSubgraphArrangements}, $(\CA_G,\one)$ is free for any graph in the list above.
	
	Conversely, assume that we have a graph $G=(N,E)$ and a multiplicity $\mu \geq \one$ on $\CA_G$ such that $(\CA_G,\mu)$ is free.
	
	1. Assume a vertex of $G$ has at least $4$ neighbors. 

 Then $G_3$ is a subgraph of $G$ and $\CA_G$ is totally non-free, by Corollary \ref{corollary: G3 subgraph not free}. 
	This implies that  all the vertices of $G$ have degree at most $3$.
	
	2. Assume that $G$ has at least $2$ cycles of length at least $3$ that share an edge.
	Since we assume that $\vert N\vert\geq 5$, the graph cannot be equal to $G_1$ or $G_2$. 
	Let $C_1,C_2$ be cycles of $G$ with a shared edge $(v_1,v_2)$, choose $w_1\in C_1,w_2\in C_2$ such that $w_1\not\in C_2,w_2\not\in C_1$. 
	Assume first that both cycles have length $3$. 
	In this case the induced graph on $\{v_1,v_2,w_1,w_2\}$ is of type $G_1$ or $G_2$. 
	Since $\vert N\vert\geq 5$ and $G$ is connected there exists a vertex $v$ that is a neighbor of at least one element of $\{v_1,v_2,w_1,w_2\}$, while not belonging to  either $C_1$ or $C_2$. 
	Taking the induced subgraph on $\{v_1,v_2,w_1,w_2,v\}$ yields one of $G_{9}$ up to $G_{19}$.
	So $\CA_G$ is again totally non-free, thanks to Proposition \ref{proposition: G3 to G19 totally non free}. 
	Assume that at least $C_1$ is of length greater than $3$. 
	Choose $v_1,v_2,w_1,w_2$ as in the first part above and an additional $w_3\in C_1$ and take the induced subgraph on $\{v_1,v_2,w_1,w_2,w_3\}$. 
	Now contract each edge at least one vertex  of which is not in $\{v_1,v_2,w_1,w_2,w_3\}$. 
	We now arrive at one of $G_{11}, G_{12}, G_{14}, G_{15}, G_{16}, G_{17}, G_{19}, G_{20}$. 
	So $\CA_G$ is totally non-free, again by Proposition \ref{proposition: G3 to G19 totally non free}.
	
	3. Assume that $G$ has at least 2 cycles of length at least $3$ without any shared edge.
	In this case we contract edges until we end up with $G_3$ as subgraph and use Corollary \ref{corollary: G3 subgraph not free}. 
	
	4. Assume $G$ has exactly one cycle $C$ of length at least $4$. 
	Note that we cannot have two vertices $v_1,v_2\in C$ being connected by an edge $e$ not belonging to $C$, 
	since otherwise we could decompose $C$ into two cycles $C_1, C_2$ of length at least 3 with shared edge $e$,
		which was already excluded in case 2.
	If $G$ is not a  cycle graph, then we can choose a vertex $v$ such that it is the neighbor of a vertex in $C$. 
	After contracting edges of $C$ until we arrive at $C_4$ and taking the subgraph induced by the reduced cycle and $v$, 
	we get one of $G_{6}, G_{12}, G_{15}, G_{17}, G_{20}$. So $\CA_G$ it is totally non-free, 
	again thanks to Proposition \ref{proposition: G3 to G19 totally non free}. 
	
	5. Assume $G$ has exactly one cycle of length $3$.
	Let $v_1, v_2, v_3$ be the unique cycle in $G$. Assume that all vertices $v_1 , v_2 , v_3$ are of
	degree $3$. Taking the subgraph induced by $\{v_1 , v_2 , v_3\}$ together with their respective third
	neighbors yields the graph $G_7$, since all neighbors are distinct (($v_1 , v_2 , v_3)$ is the only cycle in $G$). Using Proposition \ref{proposition: G3 to G19 totally non free}, we see that $\CA_{G_7}$ is totally non-free, so we can assume that the vertex $v_3$ has degree $2$. 
	If $G$ had a vertex $v$ of degree $3$ different from $v_1$ and $v_2$, then contracting all edges between $v$ and $v_1$
	(or $v$ and $v_2$, we choose whatever vertex can be reached first by following the unique path) results in a graph with a vertex of degree at least $4$. This is a contradiction as seen in case 1. Therefore, all other vertices are of degree at most $2$ and as $v_1, v_2, v_3$ constitute its only cycle, $G$ must be a path-with-triangle-graph.
	
	6. Assume that $G$ is a tree.
	If there are two vertices of degree $3$, then we contract edges until it becomes a vertex of at least degree $4$. This connected subgraph arrangement is totally non-free, as seen in case 1. So there is at most one vertex of degree $3$. If $G$ it is not a path-graph or an almost-path-graph, then it has $G_8$ as an induced subgraph and the arrangement $\CA_G$ is totally non-free, due to Proposition \ref{proposition: G3 to G19 totally non free}.
\end{proof}

The bound in 
Theorem \ref{theorem: free connected subgrapharrangement rank bigger 4} on the number of vertices is necessary since there exist free multiplicities for $\CA_{G_1}$ and $\CA_{G_2}$. We present two of them in the following result. We emphasize that the presented multiplicity on $\CA_{G_2}$ is recursively free but fails to be additively free. The given multiplicity on $\CA_{G_1}$ in turn is inductively free. 

\begin{proposition}\label{proposition: Free G1 and G2 multiplicities}
	Let $\CA\in\{\CA_{G_1},\CA_{G_2}\}$. There exists a multiplicity $\mu \ne \one$ such that $(\CA,\mu)$ is free.
	In particular, there exists a multiplicity $\mu$ such that $(\CA_{G_2},\mu)$ is recursively free but not additively free.
\end{proposition}
\begin{proof}
	We fix the following order on the hyperplanes of $\CA_{G_2}$: 
	\[H_1,H_2,H_3,H_4,H_{1234},H_{34},H_{23},H_{234},H_{14},H_{13},H_{134},H_{12},H_{124},H_{123},H_{24},\]
	and define the multiplicity $\mu_2=(2,2,2,2,2,1,1,1,1,1,1,1,1,1,1)$. Then the multiarrangement $(\CA_{G_2},\mu_2)$ is recursively free but not additively free. It follows from Remark \ref{rem:AddDelThm_multi_simple} that $(\CA_{G_2},\mu_2)$ fails to be additively free since we have $\vert((\CA_{G_2})^H,(\mu_2)^*)\vert=14$ for an arbitrary $H\in \CA_{G_2}$.
	
	For $\CA_{G_1}$ we drop the last hyperplane $H_{24}$ from the list above, but fix the same order on the rest of the hyperplanes  
	and define $\mu_1=(2, 2, 2, 1, 2, 1, 1, 1, 2, 1, 1, 1, 1, 1)$. We claim that $(\CA_{G_1},\mu_1)$  is inductively free. 
	
	The facts that $(\CA_{G_1},\mu_1)$ is inductively free and $(\CA_{G_2},\mu_2)$ is recursively free follow from the corresponding induction tables, Table \ref{G1IndFree} and Table \ref{tableRecFree}, respectively.
\end{proof}

\begin{remark}
    We note that $\CA_{G_2}$ appears in the famous counterexample by Edelmann and Reiner  \cite{edelmannreiner: counterexample orlik} to Orlik's conjecture as a non-free restriction of a free arrangement $\CB$. As a multiarrangement $(\CA_{G_2},\mu_2) = (\CB'', \kappa)$ is the Ziegler restriction of $\CB$ (recall Theorem \ref{theorem: ziegler restriction}). This gives an alternative proof of the freeness of $(\CA_{G_2},\mu_2)$. 
    Note that  $\CA_{G_2} = \CA_{G_1} \cup \{H_{24}\}$. It turns out that starting with the inductively free multiarrangement $(\CA_{G_1},\mu_1)$ and adding $H_{24}$ with multiplicity $1$ gives an inductively free multiplicity on $\CA_{G_2}$.
\end{remark}

\noindent Since every connected graph on four vertices is either $G_1, G_2, P_4, C_4, A_{3,2}$ or $ \Delta_{3,1}$ we can combine Theorem \ref{theorem: free connected subgrapharrangement rank bigger 4} and Proposition \ref{proposition: Free G1 and G2 multiplicities} to derive the following result.

\begin{corollary}
	\label{coro:GraphsWithFreeMultiplicities}
	Let $G$ be a connected graph. There exists a multiplicity $\mu$ such that the connected subgraph multiarrangement $(\CA_G,\mu)$ is free if and only if $G$ is $G_1$, $G_2$, a path-graph, a cycle-graph, an almost-path-graph, or a path-with-triangle-graph.
\end{corollary}

\section{Constant multiplicities on connected subgraph arrangements}
\label{sect:FreeConstantMultiplicities}

\subsection{Investigating $G_1$ and $G_2$}
\label{subsect:nonfreeag}

The proof in \cite[\S 6.2]{cuntzkuehne:subgrapharrangements} of the non-freeness of the connected subgraph arrangements stemming from $G_1$ up to $G_8$ depends on tedious computer calculations.
In this section, we provide a non-computational, conceptional proof giving a complete list of free constant multiplicities on the family of graphs from Theorem $\ref{theorem: FreeConnectedSubgraphArrangements}$, as well as giving a non-computational proof for the non-freeness of $\CA_{G_1}$ and $\CA_{G_2}$ in the process. Combined with Proposition \ref{proposition: G3 to G19 totally non free} we obtain a proof for the non-freeness of the simple connected subgraph arrangements stemming from $G_1$ up to $G_{20}$ free of any machine calculations. We use the following method for calculating $\LMP (2)$ of the connected subgraph arrangements. Then we utilize Lemma \ref{lemma:LMP_GMP_nonfree}.

\begin{remark}
\label{rem:alg1}
	\textbf{Algorithm to calculate $\LMP (2)$ for a simple $\CA_G$:}
	\begin{enumerate}
		\item  Determine the number of all types (i.e.~path-graph, almost-path-graph, etc.) of connected
		induced subgraphs $G[S]$ of $G = (N, E)$, where $S \subseteq N$ and $\vert S\vert \geq 2$. 
		
		\item For every subgraph $G[S]$ from step $(1)$ we determine the number of triples of the form  $(G[S],G[I],G[J])$,
		where $G[I],G[J]$ are connected induced subgraphs of $G[S]$ with $I \, \dot{\cup} \, J = S$. Adding up the number of all such triples for all connected induced subgraphs $G'$ from (1) provides  an invariant $\nu_1$. Using Lemma \ref{lemma: connected subgraph arrangements are locallyA2} we see that for a fixed graph $G'$ this is just the number of rank $2$ localizations with $N(G')=A_{i_3}$, with the notation as in the lemma. 
		
		\item Calculate $\nu_2:=\binom{\vert\CA_G\vert}{2}$. If all rank $2$ localizations were to consist of exactly two elements, then $\nu_2$ would be the number of all rank $2$ localizations.
		
		\item We have $\LMP (2)= \nu_2-\nu_1$. To see this, a rank $2$ localization with two elements is free with exponents $(1,1)$, while a rank $2$ localization with three elements is free with exponents $(1,2)$. So if instead of having three rank $2$ localizations with two elements we have one rank $2$ localization with three elements, then $\LMP (2)$ gets decreased by $1$ (since we assume for $\nu_2$ that all localizations contain exactly two elements). Since $\nu_1$ is the number of rank $2$ localizations with three elements, $\nu_2-\nu_1$ coincides with $\LMP (2)$.
	\end{enumerate}
\end{remark}

We demonstrate how this algorithm works in the cases where $G\in\{G_1, G_2\}$.

\begin{example}
	We begin with $G=G_1$.
	\begin{itemize}
		\item $5\cdot P_2, 2\cdot P_3, 2\cdot C_3, 1\cdot G_1$
		\item $G'=P_2: 1\cdot (P_2,P_1,P_1)$ \\
		$G'=P_3: 2\cdot (P_3,P_2,P_1)$\\
		$G'=C_3: 3\cdot (C_3,P_2,P_1)$\\
		$G'=G_1: 2\cdot (G_1,C_3,P_1), 2\cdot (G_1,P_3,P_1), 2\cdot (G_1,P_2,P_2)$\\
		$\nu_1=5\cdot 1+2\cdot 2+2\cdot 3+1\cdot 6=21$
		\item $\nu_2=\binom{14}{2}=91$
		\item $\LMP (2)=91-21=70$
	\end{itemize}
	For $G=G_2$ we get the following result.
	\begin{itemize}
		\item $6\cdot P_2, 4\cdot C_3, 1\cdot G_1$
		\item $G'=P_2: 1\cdot (P_2,P_1,P_1)$ \\
		$G'=C_3: 3\cdot (C_3,P_2,P_1)$\\
		$G'=G_2: 4\cdot (G_1,C_3,P_1), 3\cdot (G_1,P_2,P_2)$\\
		$\nu_1=6\cdot 1+4\cdot 3+1\cdot 7=25$
		\item $\nu_2=\binom{15}{2}=105$
		\item $\LMP (2)=105-25=80$
	\end{itemize}
\end{example}

In view of Lemma \ref{lemma:LMP_GMP_nonfree}, 
we derive the non-freeness of $\CA_{G_1}$ and $\CA_{G_2}$ without any machine 
calculations as follows.

\begin{proposition}\label{proposition: SimpleG1andG2notFree}
	Let $G\in\{G_1,G_2\}$, then $\CA_{G}$ is not free.
\end{proposition}
\begin{proof}
	Since $|\CA_{G_1}| = 14$, a quadruple of potential exponents for $\CA_{G_1}$ has to be $(1,4,4,5)_\leq$, or less balanced. Calculating $\GMP (2)$ for this tuple gives $\GMP(2)=69$. Since exponents that are less balanced afford a smaller $\GMP (2)$ and since $\LMP (2)=70$, 
	we have $\LMP (2)\ne\GMP (2)$ and so $\CA_{G_1}$ is not free, by Lemma \ref{lemma:LMP_GMP_nonfree}.
	
	Since $|\CA_{G_1}| = 15$, choose potential balanced  exponents $(1,4,5,5)_\leq$ for $\CA_{G_2}$ to maximize $\GMP (2)$. Using this quadruple, we get $\GMP (2)=79$. Since exponents that are less balanced yield a smaller value for $\GMP (2)$ and since $\LMP (2)=80$, we have $\LMP (2)\ne\GMP (2)$ and so $\CA_{G_2}$ can't be free, thanks to Lemma \ref{lemma:LMP_GMP_nonfree}.
\end{proof}

\begin{remark}
	Propositions \ref{proposition: G3 to G19 totally non free} and 
	\ref{proposition: SimpleG1andG2notFree} combined give a non-computational, conceptional  proof for the non-freeness of the simple arrangements $\CA_{G_i}$, for $i\in\{1,2,\dots, 20\}$. 
\end{remark}
			
\noindent Next we modify the algorithm from Remark \ref{rem:alg1} to calculate $\LMP (2)$ when $\CA_G$ gets equipped with a constant multiplicity $\mu = c \one$ for $c \ne 1$.

\begin{remark}
\label{rem:alg2}
	\textbf{Algorithm to calculate $\LMP (2)$ for a constant multiplicity $\mu\equiv c$:}
	\begin{enumerate}
		\item Determine the number of all types of connected induced subgraphs of $G$ with at least $2$ vertices (i.e.~path-graph, almost-path-graph, etc.). 
		\item For every subgraph $G'$ from step (1) we determine the number of triples $(G',G[I],G[J])$, where $G[I],G[J]$ are connected induced subgraphs of $G'$ with $N(G[I])\cup N(G[J])=N(G')$. Let $\nu_1$ be the sum of the number of all such triples for all connected induced subgraphs $G'$ from part (1).
		\item Calculate $\nu_2:=\binom{\vert\CA_G\vert}{2}$. 
		\item Since $\nu_1$ is the number of rank $2$ localizations with three elements, the number of rank $2$ localizations with two elements is equal to $\nu_3:=\nu_2-3 \nu_1$, because if there is one rank $2$ localization with three elements, then the hyperplanes of this localization could have potentially formed three rank $2$ localizations with two elements. 
		\item Using Theorem \ref{WakaTheo} we derive the exponents of all rank $2$ localizations with three elements (which depends on the parity of $c$), while localizations of rank $2$ with two elements are free with exponents $(c,c)$. We get $$
		\LMP (2)=\begin{cases}
			\nu_1\cdot\left(\frac{3c}{2}\right)^2+\nu_3  c^2, & \text{ (if }c\text{ is even),}\\
			\nu_1\cdot\left(\frac{3c-1}{2}\cdot \frac{3c+1}{2}\right)+\nu_3  c^2, & \text{ (if }c\text{ is odd).}
		\end{cases}
		$$
	\end{enumerate}
\end{remark}

We now use Lemma \ref{lemma:LMP_GMP_nonfree} to extend Proposition \ref{proposition: SimpleG1andG2notFree} to arbitrary constant multiplicities.

\begin{proposition}
	\label{prop:G_1_2ConstNotFree}
	For $G\in\{G_1,G_2\}$ and $\mu\equiv c \geq 1$ a constant multiplicity on $\CA_G$,  $(\CA_G,\mu)$ is not free.
\end{proposition}

\begin{proof}
	The case $c = 1$ is Proposition \ref{proposition: SimpleG1andG2notFree}. So let $c\geq 2$. 
	From the calculation in the example we know that:
	$G_1:$ $\nu_1=21,\nu_2=91$ and
	$G_2:$ $\nu_1=25,\nu_2=105$, which gives $\nu_3=28$ for $G_1$ and $\nu_3=30$ for $G_2$. 
	So for $G_1$ we get 
	$$\LMP(2)=\begin{cases}
		21\cdot\left(\frac{3c}{2}\right)^2+28  c^2=75.25c^2, & \text{if $c$ is  even.}\\
		21\cdot\left(\frac{3c-1}{2}\cdot \frac{3c+1}{2}\right)+28  c^2=75.25c^2-5.25, & \text{if $c$ is odd,}
	\end{cases}$$
	and for $G_2$ we get
	$$\LMP(2)=\begin{cases}
		25\cdot\left(\frac{3c}{2}\right)^2+30  c^2=86.25c^2, & \text{if $c$ is even.}\\
		25\cdot\left(\frac{3c-1}{2}\cdot \frac{3c+1}{2}\right)+30  c^2)=86.25c^2-6.25, & \text{if $c$ is odd.}
	\end{cases}$$
	If we assume perfectly balanced exponents for $\CA_{G_1}$ (assuming $14c\equiv 0\mod 4$) we get the set of exponents $(\frac{14c}{4},\frac{14c}{4},\frac{14c}{4},\frac{14c}{4})_\leq$ which results in $\GMP (2)=73.5c^2$.\\
	If we assume perfectly balanced exponents for $\CA_{G_2}$ (assuming $15c\equiv 0\mod 4$) we get the set of exponents $(\frac{15c}{4},\frac{15c}{4},\frac{15c}{4},\frac{15c}{4})_\leq$ which results in $\GMP (2)=84.375c^2$.\\
	Since $c\geq 2$, in any case,  we have $\LMP (2)>\GMP (2)$ and thus $(\CA_{G_2},\mu)$ can't be free, by Lemma \ref{lemma:LMP_GMP_nonfree}.
\end{proof}

\subsection{Constant multiplicities on a free $\CA_{G}$} In this subsection, we inspect free constant multiplicities on a free $\CA_G$. We give a complete list  when  $(\CA_G,\mu)$ is free for $G\in\{P_n,A_{n,k},\Delta_{n,k}, C_n\}$ and 
$\mu\equiv c \geq 1$.

First we consider the case $G=P_n$.
Terao showed in \cite{terao:free coxeter multiarrangements} that every constant multiplicity on a Coxeter arrangement is free. Since $\CA_{P_n}$ is the Coxeter arrangement of type $A_n$, we derive the freeness of $(\CA_{P_n},\mu)$ 
for $\mu\equiv c \geq 1$.

Now suppose that $G\neq P_n$. We start with $G=C_3$ and follow the idea described in Lemma \ref{lemma:LMP_GMP_nonfree} during every proof.

\begin{proposition}\label{proposition: C3FreeConstantMultiplicities}
	Let $G=C_3$ and $\mu\equiv c\geq 1$ is  a constant multiplicity on $\CA_G$. Then $(\CA_G,\mu)$ is free if and only if $c\in\{1,3\}$.
\end{proposition}

\begin{proof}
	For the forward implication, using the algorithm above, we get $\nu_1=1\cdot 3+3\cdot 1=6,\nu_2=\binom{7}{2}=21,\nu_3=21-3\cdot 6=3$. This gives 
	$$\LMP (2)=\begin{cases}
		6\cdot\left(\frac{3c}{2}\right)^2+3  c^2=16.5c^2 & \text{ (if }c\text{ is even), or }\\
		6\cdot\left(\frac{3c-1}{2}\cdot \frac{3c+1}{2}\right)+3  c^2=16.5c^2-1.5 & \text{ (if }c\text{ is odd).}
	\end{cases}$$    
	Assuming balanced exponents we get 
	$$\GMP (2)=\begin{cases}
		16\frac{1}{3}c^2, & \text{if }7c\equiv 0\mod 3\text{ or}\\
		16\frac{1}{3}c^2-\frac{1}{3}, & \text{else.}
	\end{cases}$$
	For $c\geq 4$ we get that $\LMP (2)>\GMP (2)$. For $c=3$ we get equality and for $c=2$ we get $\LMP (2)=66$ and depending on whether the assumed exponents are $(3,5,6)_\leq,(4,4,6)_\leq$, or $(4,5,5)_\leq$, a value of $63, 64$, or $65$, respectively for $\GMP (2)$.
	So if $c\not\in\{1,3\}$ we have $\LMP (2)>\GMP (2)$ and thanks to Lemma \ref{lemma:LMP_GMP_nonfree}, we deduce non-freeness.
	
	Next we address the reverse implication. 
 The freeness of $(\CA_G,\one)$ follows from Theorem \ref{theorem: FreeConnectedSubgraphArrangements}.  
In Table \ref{C3ThreeIndFree}, 
we present an induction table for $(\CA_G,3 \one)$ showing that the latter is actually inductively free.
\end{proof}

\begin{proposition}
	Let $G\in\{C_n,\Delta_{n,k},A_{n,k}\}$. Suppose $\rank\CA_G\geq 4$ and $\mu\equiv c\geq 1$ a constant multiplicity on $\CA_G$. Then $(\CA_G,\mu)$ is free if and only if $c=1$.
\end{proposition}

\begin{proof}
	First let $G\in\{C_n,\Delta_{n,k}\}$. Since $G$ either has $C_3$ as an induced subgraph or becomes $C_3$ after contracting edges, we can use Lemma \ref{lemma:ind_subgraph contract localization} and Proposition \ref{proposition: C3FreeConstantMultiplicities} to focus on the cases where $c\in\{1,3\}$. If $c=1$ we know that $(\CA_G,\mu)$ is free. So let $\mu\equiv 3$ and $G=C_4$.\\
	Calculating $\LMP (2)$ we get $\nu_1=1\cdot (4+2)+ 4\cdot 2+4\cdot 1=18, \nu_2=\binom{13}{2}=78,\nu_3=78-3\cdot 18=24$ and so 
	\[\LMP (2)=18\cdot\left(\frac{3\cdot3-1}{2}\cdot\frac{3\cdot3+1}{2}\right)+24\cdot 9=576.\]
	Since $\CA_{C_4}$ consists of $13$ hyperplanes, assuming balanced exponents $(9,10,10,10)_\leq$, we get $\GMP (2)=570$. So $\LMP (2)>\GMP (2)$. Using Theorem \ref{MixedEquality}, we see that $\CA_G$ fails to be free. For a given $C_n$ (for  $n > 4$) we can contract edges until we end up with $C_4$. The result follows.
	
	Let $G=\Delta_{3,1}$ and note that every $\Delta_{n,k}\neq C_3$ can get transformed into $G$ by contracting edges. If $\mu\equiv 3$ we get $\LMP (2)=489$ and assuming exponents $(9,9,9,9)_\leq$ we get $\GMP (2)=486$. So only the simple multiplicity can be free.
	
	Let $G=A_{3,2}$. Let $\mu\equiv c$, then we get
	$$
	\GMP (2)=\begin{cases}
		45.375c^2, & \text{if }11c \equiv 0\mod 4,\\
		45.375c^2-\frac{1}{2}, & \text{if }11c \equiv 2\mod 4,\\
		45.375c^2-\frac{3}{8}, & \text{else.}\\
	\end{cases}
	$$
	Moreover,  $\LMP (2)=46c^2$ (if $c$ is even) and $46c^2-3$ (if $c$ is odd).
	Once more we see that $\LMP (2)\neq \GMP (2)$ and so, by Theorem \ref{MixedEquality}, $(\CA,\mu)$ fails to be free. Since every $A_{n,k}$ with $n\geq 4$ can be transformed into $A_{3,2}$ after contracting edges, the result follows from Lemma \ref{lemma:ind_subgraph contract localization}.  
\end{proof}

Combining the last two results we get the following.

\begin{corollary}
	\label{coro:ConstMultFree}
	Let $G\in\{P_n,A_{n,k},\Delta_{n,k},C_n\}$ , $\mu\equiv c$ $(c>1)$. Then $(\CA_G,\mu)$  is free if and only if $G=P_n$  or $G=C_3$  and $c=3$.
\end{corollary}

\section{Classifying free multiplicities on connected subgraph arrangements}
\label{sect:ClassificationSomeCases}
\noindent Since connected subgraph arrangements are locally $A_2$ and admit a positive system of defining equations, 
the extension techniques introduced by Yoshinaga in \cite{yoshinaga:extendable} are well suited 
to attack the problem of classifying free multiplicities on connected subgraph arrangements and their subarrangements. 

We define a subarrangement of $\CA_{C_3}$ for which we give a complete list of free multiplicities. 
This classification was achieved previously by DiPasquale and Wakefield in \cite{dipasquale: X3 moduli freeness}, 
where this arrangement appeared in a different context, using homological algebra methods.

\begin{defn}
	We define the \emph{deleted $C_3$ arrangement} $\mathscr{D}$ as $\mathscr{D}:=\CA_{C_3}\backslash\{H_{123}\}$. The defining polynomial of $\mathscr{D}$ is given by 
	\begin{equation}
		\label{eq:D}
		Q(\mathscr{D})=x_1x_2x_3(x_1+x_2)(x_1+x_3)(x_2+x_3).
	\end{equation}    
\end{defn}

To prove the freeness of a certain family of multiplicities on $\mathscr{D}$ and to identify some free multiplicities on $\CA_{C_3}$ at the end of this section, we require the following lemma.

\begin{lemma}\label{lemma: special B2 multiplicities}
	Let $(\CA,\mu)=(\CA(B_2),\mu)$ be the multiarrangement $$Q(\CA(B_2),\mu)=x_1^ax_2^b(x_1+x_2)^c(2x_1+x_2)^d$$ with multiplicity $\mu=(a,b,c,d)$.
	\begin{enumerate}
		\item If $\mu=(2k,1,2k,1)$, then $(\CA,\mu)$ is free with exponents $(2k+1,2k+1)$.
		\item If $\mu=(2k+1,1,2k+1,1)$, then $(\CA,\mu)$ is free with exponents $(2k+1,2k+3)$.
		\item If $\mu=(k,1,k,2k-4)$ or $\mu=(k,2k-4,k,1)$, then $(\CA,\mu)$ is free with exponents $(2k-2,2k-1)$.
		\item If $\mu=(k+1,1,k+1,2k-4)$ or $\mu=(k+1,2k-4,k+1,1)$, then $(\CA,\mu)$ is free with exponents $(2k-1,2k)$.
	\end{enumerate}
\end{lemma}

\begin{proof}
	(1) and (2): We give an explicit basis for the module of derivations $D(\CA,\mu)$ in both cases. We utilize the linear isomorphism from $\BBC^2$ to itself, defined by $x_1\mapsto x_2$ and $x_2\mapsto x_1-x_2$, sending 
 $\{x_1,x_2,x_1+x_2,2x_1+x_2\}$ to $\{x_1,x_2,x_1+x_2,x_1-x_2\}$. Thus it is sufficient to calculate a basis for $D(\mathscr{B},\nu)$, where $Q(\mathscr{B},\nu)=x_1^cx_2^a(x_1+x_2)^d(x_1-x_2)^b$.\\
    For (1), we have $Q(\mathscr{B},\nu)=x_1^{2k}x_2^{2k}(x_1+x_2)^1(x_1-x_2)^1$ and a basis for $D(\mathscr{B},\nu)$ is given by $$\theta_1=x_1^{2k}x_2dx_1+x_1x_2^{2k}dx_2\quad\text{ and }\quad\theta_2= x_1^{2k+1}dx_1+ x_2^{2k+1}dx_2.$$
    For (2), we have $Q(\mathscr{B},\nu)=x_1^{2k+1}x_2^{2k+1}(x_1+x_2)^1(x_1-x_2)^1$ and a basis for $D(\mathscr{B},\nu)$ is given by $$\theta_1=x_1^{2k+1}dx_1+x_2^{2k+1}dx_2\quad\text{ and }\quad\theta_2= (x_1^2x_2^{2k+1}-x_2^{2k+3})dx_2.$$
    Verifying that $\theta_1, \theta_2\in D(\mathscr{B},\nu)$ in both cases and checking the determinant condition in Theorem \ref{theorem: saitos criterion} (which gives $\det(p_{ij})=x_1^{2k} x_2^{2k + 2} - x_1^{2k + 2} x_2^{2k}=(x_1 - x_2) (x_1 + x_2) (-x_1^{2k}) x_2^{2k}$ in the first case and $\det(p_{ij})=(x_1 - x_2) (x_1 + x_2) x_1^{2k + 1} x_2^{2k + 1}$ in the second case) shows that in both cases $\{\theta_1,\theta_2\}$ is a basis for $D(\mathscr{B},\nu)$, which completes the proof.\\
	(3) and (4): Start with $\mu'=(k,0,k,2k-4)$ resp.~$\mu'=(k+1,0,k+1,2k-4)$ and use Theorem \ref{WakaTheo} to see that $(\CA,\mu')$ is free with exponents $(2k-2,2k-2)$ resp.~$(2k-1,2k-1)$. So $(\CA,\mu)$ has to be free with exponents $(2k-2,2k-1)$ resp.~$(2k-1,2k)$.\\
    The argument is the same for the other two multiplicities, but we delete $\ker(2x_1+x_2)$ instead.
\end{proof}

\subsection{Free multiplicities on $\mathscr{D}$.}
We first introduce a special family of multiplicities $\mu$ such that $(\mathscr{D},\mu)$ is a free multiarrangement. Then we proceed to show that there are no other free multiplicities on $\mathscr{D}$.

\begin{proposition}\label{Proposition: FreeEvenAnd1Multiplicities}
	Let $\mu=(2k,2k,2k,1,1,1)$ (for $k\in\mathbb{N}_{\geq 1}$) be a multiplicity on $\mathscr{D}$ where the order on $\mathscr{D}$ is as in \eqref{eq:D}. Then $(\mathscr{D},\mu)$ is free with exponents $(2k+1,2k+1,2k+1)$.
\end{proposition}

\begin{proof}
	The subarrangement $\mathscr{B}_0$ of $\mathscr{D}$ defined by $Q(\mathscr{B}_0,\mu_0)=x_1^{2k}x_2^{2k}x_3^{2k}$ is boolean and therefore free with exponents $(2k,2k,2k)$. Let $Q(\mathscr{B}_1,\mu_1):=x_1^{2k}x_2^{2k}x_3^{2k}(x_1+x_2)$ and $H_1=\ker(x_1+x_2)$. Then we get $\mathscr{B}_1^{H_1}=\{\ker(x_1),\ker(x_2)\}$. Using Proposition \ref{ATWEulerProp}, we get $\mu_1^*=(2k,2k)$ for the corresponding Euler-multiplicity. Therefore, $(\mathscr{B}_1^{H_1},\mu_1^*)$ is free with exponents $\exp(\mathscr{B}_1^{H_1},\mu_1^*)=(2k,2k)$.
	Using Theorem \ref{thm:add-del} we deduce that $(\mathscr{B}_1,\mu_1)$ is free with exponents $(2k,2k,2k+1)$. 
	Now let $Q(\mathscr{B}_2,\mu_2):=x_1^{2k}x_2^{2k}x_3^{2k}(x_1+x_2)(x_1+x_3)$ and $H_2=\ker(x_1+x_3)$, then we get $\mathscr{B}_2^{H_2}=\{\ker(x_1),\ker(x_2),\ker(x_1-x_2)\}$ and Proposition \ref{ATWEulerProp} gives $\mu_2^*=(2k,2k,1)$. Therefore, $(\mathscr{B}_2^{H_2},\mu_2^*)$ is free with exponents $\exp(\mathscr{B}_2^{H_2},\mu_2^*)=(2k,2k+1)$, so $(\mathscr{B}_2,\mu_2)$ is free with exponents $(2k,2k+1,2k+1)$. 
	Finally, restricting $(\mathscr{D},\mu)$ to $H_3=\ker(x_2+x_3)$ gives  $\mathscr{D}^{H_3}=\{x_1,x_2,x_1+x_2,x_1-x_2\}$ with Euler-multiplicity $\mu^*=(2k,2k,1,1)$. In the proof of Lemma \ref{lemma: special B2 multiplicities} we show that $(\mathscr{D}^{H_3},\mu^*)$ is free with exponents $\exp(\mathscr{D}^{H_3},\mu^*)=(2k+1,2k+1)$. Using Theorem \ref{thm:add-del} once more completes the proof.   
\end{proof}

The proof above actually shows that  $(\mathscr{D},\mu)$ is inductively free
with $\mu$ as in the statement of the proposition. 
Next, we prove that there are no other free multiplicities on $\mathscr{D}$. 

\subsection{Non-Free multiplicities on $\mathscr{D}$} 

The following Lemma shows that the non-trivial multiplicities as in Proposition \ref{Proposition: FreeEvenAnd1Multiplicities} all have to coincide.

\begin{lemma}\label{lemma: 3evenMustBeEqual}
	Let $Q(\mathscr{D},\mu)=x_1^ax_2^bx_3^c(x_1+x_2)^d(x_1+x_3)^e(x_2+x_3)^f$ with multiplicity $\mu=(a,b,c,d,e,f), a\geq b\geq c$. If $\mu=(even,even,even,1,1,1)$, then
		$(\mathscr{D},\mu)$ is free if and only if $a=b=c$.
\end{lemma}
\begin{proof}
	The reverse implication follows from Proposition \ref{Proposition: FreeEvenAnd1Multiplicities}.
 
	For the forward implication assume that $(\mathscr{D},\mu)$ is free.
	For some $X\in L(\CA)_2$ and the corresponding rank $2$ localization $\CA_X=\{H_1,H_2,H_3\}$ or $\CA_X=\{H_1,H_2\}$ which is free with $\exp(\CA_X,\mu_X)=(d_1^X,d_2^X)$, we write its summand $d_1^Xd_2^X$ in $\LMP (2)$ as $[\mu(H_1),\mu(H_2),\mu(H_3)]$ or $[\mu(H_1),\mu(H_2)]$ respectively. Using this notation, $\LMP (2)$ of $(\mathscr{D},\mu)$ equals
	\[\LMP (2)=[a,b,d]+[a,c,e]+[b,c,f]+[a,f]+[b,e]+[c,d]+[d,e]+[d,f]+[e,f].\]
	With our assumption $d=e=f=1$, we get 
	\[\LMP (2)=[a,b,1]+[a,c,1]+[b,c,1]+a+b+c+3,\]
	which (because of $a\geq b\geq c$ and Theorem \ref{WakaTheo}) equals 
	\[a(b+1)+a(c+1)+b(c+1)+a+b+c+3= 3 a + 2 b + c + a b + a c +  b c +3.\]
	Define the multiplicity $\mu_1=(a,b,c,0,0,0)$ on $\mathscr{D}$, then $(\mathscr{D},\mu_1)$ is free with exponents $(a,b,c)$ since it is boolean. Assuming that $(\mathscr{D},\mu)$ is free, the arrangement has exponents equal to one of $(a+3,b,c)_\leq,(a+2,b+1,c)_\leq,\dots,(a,b,c+3)_\leq$, thanks to {\cite[Rem.~2.8]{hogeroehrleschauenburg:free}}. Writing $E:=ab+ac+bc$ this gives  $\GMP (2)$ as follows:
	
	\begin{center}
		\tiny
		\begin{tabular}{|c| c| c| c| c | c |} 
			\hline
			Exponents & $(a+3,b,c)_\leq$ & $(a,b+3,c)_\leq$ & $(a,b,c+3)_\leq$ & $(a+2,b+1,c)_\leq$ & $(a+2,b,c+1)_\leq$ \\ 
			\hline
			$\GMP (2)$ & $E + 3 b + 3 c$ & $E + 3 a + 3 c$ & $E + 3 a + 3 b$ & $E + a + 2 b + 3 c + 2$ & $E + a + 3 b + 2 c + 2$ \\ 
			\hline
		\end{tabular}
		\begin{tabular}{|c| c| c| c| c | c |} 
			\hline
			Exponents & $(a+1,b+2,c)_\leq$ & $(a,b+2,c+1)_\leq$ & $(a+1,b,c+2)_\leq$ & $(a,b+1,c+2)_\leq$ & $(a+1,b+1,c+1)_\leq$ \\ 
			\hline
			$\GMP (2)$ & $E + 2 a + b + 3 c + 2$ & $E + 3 a + b + 2 c + 2$ & $E + 2 a + 3 b + c + 2$ & $E + 3 a + 2 b + c + 2$ & $E + 2 a + 2 b + 2 c + 3$ \\ 
			\hline
		\end{tabular}
	\end{center}
	Since each of $a,b,c$ is even by assumption, we get that in all but the last case (where $\exp(\mathscr{D},\mu)=(a+1, b+1, c+1)_\leq$), $\GMP (2)$ is even. But since 
	\[\LMP (2)=E+ 3 a + 2 b + c +3\]
	is odd, we see that $\LMP (2)\neq \GMP (2)$ in those cases. In particular, the arrangement cannot be free because of Theorem \ref{MixedEquality}. So let us assume $\exp(\mathscr{D},\mu)=(a+1,b+1,c+1)_\leq$ which gives $\GMP (2)=E + 2 a + 2 b + 2 c +  3$. 
	But if we do not have $a=b=c$, then $a>b$ or $b>c$. This implies $a>c$ and we have
	\[\GMP (2)= 2 a + 2 b + 2 c + E +  3 < 3 a + 2 b + c + E +3=\LMP (2).\]
	So using Theorem \ref{MixedEquality} again, we see that $(\mathscr{D},\mu)$ cannot be free and we have $a=b=c$. 
\end{proof}

Theorem \ref{theorem: Yoshinaga extendable theorem} allows us to focus on the cases where $\mu$ does not satisfy \eqref{eq:star} as follows.

\begin{corollary}\label{corollary: DnotFreeWithStar}
	Let $\mu$ be a multiplicity on $\mathscr{D}$ that satisfies \eqref{eq:star}. Then $(\mathscr{D},\mu)$  is not free. In particular, $\mathscr{D}$ is not free. 
\end{corollary}

\begin{proof}
	By \cite[Prop.~3.1]{cuntzkuehne:subgrapharrangements}, $\CA_{C_3}$ is free with exponents $(1,3,3)$. Let $H=H_{123}$ and note that $\CA_{C_3}^H$ is free with exponents $(1,2)$, since $\vert\CA_{C_3}^H\vert=3$. According to Theorem \ref{thm:add-del},  
 $\CA_{C_3}\backslash\{H\}=\mathscr{D}$ is not free.
	By Theorem \ref{theorem: Yoshinaga extendable theorem}, $(\mathscr{D},\mu)$ is free if and only if $\CE(\mathscr{D},\mu)$ free. However, $\CE(\mathscr{D},\mu)_X=\mathscr{D}$ for $X=H_1\cap H_2\cap H_3$. So $\CE(\mathscr{D},\mu)$ cannot be free, by Theorem \ref{theorem: Multilocalizations are free}.
\end{proof}

\noindent We use a similar argument when \eqref{eq:star} is not satisfied. For this, we define a free extension for the multiarrangement $x^ay^b(x+y)^c$ when $(a,b,c)=(even, even, odd)$. With this tool in hand, we construct an extension of $(\mathscr{D},\mu)$ which is locally free along $H_4$ for an arbitrary multiplicity $\mu$. This enables us to use Corollary \ref{Cor: LocalFreeChaPolFactor} to show that $(\mathscr{D},\mu)$ is not free.  

Next we consider a class of extensions of a given multiarrangement $(\CA,\mu)$ which generalizes the one from Definition \ref{definition: Yoshinga extension}.

\begin{defn}
	Let $(\CA,\mu)$ be a multiarrangement in $\mathbb{K}^\ell$. 
	Define the extension $\CE_\sigma(\CA,\mu)$ in $\mathbb{K}^{\ell+1}$ depending on another fixed multiplicity $\sigma:\CA\to \mathbb{Z}$ as follows:
	\[\CE_\sigma(\CA,\mu)=\{x_{\ell+1}=0\}\cup\left\{\alpha_H-kx_{\ell+1}=0 \ \middle| \ k\in\mathbb{Z}, -\frac{\mu(H)-1}{2}+\sigma(H)\leq k\leq \frac{\mu(H)}{2}+\sigma(H)\right\} .\]
	Note that $\CE_0(\CA,\mu)=\CE(\CA,\mu)$ as in Definition \ref{definition: Yoshinga extension}. 
\end{defn}

We occasionally denote a multiplicity $\mu$ by $\mu=(\mu(H_1),\mu(H_2),\dots,\mu(H_n))$, likewise for $\sigma$.
We use this notation in the following lemma.

\begin{lemma}\label{lemma: localFreenessWithoutStar}
	Let $Q(\CA,\mu)=x_1^{m_1}x_2^{m_2}(x_1+x_2)^{m_3}$ with $m_1\leq m_2$. Let $\sigma=(0,0,1)$.
	\begin{itemize}
		\item[(i)] If $\mu = (m_1,m_2,m_3)=(even,even,odd)$, then $\CE_\sigma(\CA,\mu)$ is free.
		\item[(ii)] If $\mu =  (m_1,m_2,m_3)=(even,odd,odd)$ or $(odd,even,odd)$, then $\CE_\sigma(\CA,\mu)$ is free.
	\end{itemize}
\end{lemma}
\begin{proof}
	The idea of the proof is to increase the multiplicity of $H_{12}$ by $1$ which results in a multiplicity $\mu_1=(m_1,m_2,m_3+1)$ such that \eqref{eq:star} is satisfied and thus $\CE(\CA,\mu_1)$ is a  free extension of $(\CA,\mu_1)$. 
	Now we remove a hyperplane $H$ so that $\CE(\CA,\mu_1)\backslash\{H\}=\CE_\sigma(\CA,\mu)$ is free, thanks to  Theorem \ref{thm:add-del}. We demonstrate this in one case while the other cases are handled analogously.
	
	For (i) let $\mu_1=(m_1,m_2,m_3+1)=(even,even,even)$. The extended arrangement $\CA_1:= \CE_0(\CA,\mu_1)$ is free, by Theorem \ref{theorem: Yoshinaga extendable theorem}. Note that for $\sigma=(0,0,1), H=\ker(x_1+x_2+(\frac{m_3-1}{2})x_3)$ and $\CA_2:= \CE_{\sigma}(\CA,\mu)$, we have $\CA_2=\CA_1\backslash \{H\}$. So we start at the free arrangement $\CA_1$, fix the hyperplane $H$, and use Theorem \ref{thm:add-del} to derive the freeness of $\CA_2$. Note that the exponents of $\CA_1$ can be derived by using Theorems \ref{WakaTheo} and \ref{theorem: ziegler restriction}. The following case by case investigation is necessary due to {\cite[Lem.~3.1]{yoshinaga:extendable}}.
	
	Assume $m_2-m_1\leq m_3 < m_1+m_2$.
	Then we have $\exp(\CA_1)=(1,\frac{m_1+m_2+m_3+1}{2},\frac{m_1+m_2+m_3+1}{2})$.
	Note that it is sufficient to show that $\vert \CA_1^{H}\vert=1+\frac{m_1+m_2+m_3+1}{2}$, since  $\CA_1^{H}$ is a simple arrangement of rank $2$ and so is free with exponents $(1,\frac{m_1+m_2+m_3+1}{2})$, satisfying the conditions of Theorem \ref{thm:add-del}. We count all elements $X\in L(\CA_1)_2$ with $X\subset H$ by listing all localizations $\CA_X$. There is a canonical bijection between the set containing those elements and the set of hyperplanes of $\CA_1^H$. We identify the hyperplanes of $\CE(\CA,\mu_1)$ with their linear forms.
			We have a single  
			localization with more than $3$ elements: $$\left\{H,x_1+x_2+\left(\frac{m_3-3}{2}\right)x_3,\dots,x_1+x_2-\frac{m_3+1}{2}x_3,x_3\right\}$$
			and a total of 
			$\frac{m_1+m_2-(m_3+1)}{2}$ localizations with $3$ elements
			\begin{multline*}
				\left\{H,x_1+\left(\frac{m_1-2}{2}\right)x_3,x_2-\left(\frac{m_1-(m_3+1)}{2}\right)x_3\right\}, \\
    \left\{H,x_1+\left(\frac{m_1-4}{2}\right)x_3,x_2-\left(\frac{m_1-(m_3+3)}{2}\right)x_3\right\}, \\ \dots, \left\{H,x_1-\left(\frac{m_2-(m_3+1)}{2}\right)x_3,x_2+\left(\frac{m_2-2}{2}\right)x_3\right\}.
			\end{multline*}
		
		The rank $2$ localization with more than $3$ elements contains $m_3+1$ elements which are not equal to $H$ while all of the localizations with $3$ elements combined contain $m_1+m_2-m_3-1$ hyperplanes distinct from $H$. Subtracting the number of hyperplanes included in these localizations from the total number of hyperplanes in $\CA_1$ shows that there are $$(m_1+m_2+(m_3+1))-(m_3+1)-(m_1+m_2-m_3-1)=m_3+1$$ hyperplanes $H'\in \CA_1$ such that $\vert(\CA_1)_{H\cap H'}\vert=2$.
	In conclusion we have $$1+\frac{m_1+m_2-(m_3+1)}{2}+(m_3+1)=1+\frac{m_1+m_2+m_3+1}{2}$$ hyperplanes in $\vert\CA_1^{H}\vert$. This implies that $\exp(\CA_1^{H})=(1,\frac{m_1+m_2+m_3+1}{2})$. By Theorem \ref{thm:add-del}, $\CA_2$ is free with exponents $\exp(\CA_2)=(1,\frac{m_1+m_2+m_3-1}{2},\frac{m_1+m_2+m_3+1}{2})$.
	The cases where $m_3+1<m_2-m_1+1$ and $m_3+1\geq m_1+m_2+1$ and part (ii) are handled analogously.
\end{proof}

\begin{theorem}\label{Theorem: AllDMultiplicities}
	Let $(\mathscr{D},\mu)$ be the multiarrangement $$Q(\mathscr{D},\mu)=x_1^ax_2^bx_3^c(x_1+x_2)^d(x_1+x_3)^e(x_2+x_3)^f$$ with $a,b,c,d,e,f \geq 1$. Then
	$(\mathscr{D},\mu)$ is free if and only if  $\mu=(2k,2k,2k,1,1,1)$ for some $k\in\mathbb{N}_{\geq 1}$ where the order of the hyperplanes is as above.
\end{theorem}

\begin{proof}
	The reverse implication follows from Proposition \ref{Proposition: FreeEvenAnd1Multiplicities}.\\
	For the forward implication note that if at least two of the values $a,b,c$ are odd or all of $d,e,f$ are even, then $(\mathscr{D},\mu)$ is not free,  by Corollary \ref{corollary: DnotFreeWithStar}. Consequently, we can assume  (after possibly permuting some basis vectors) that at least $d$ is odd and $a,b$ are even so \eqref{eq:star} is not satisfied. If $\mu=(even,even,even,1,1,1)$ and $\mu\neq(2k,2k,2k,1,1,1)$, then we use Lemma \ref{lemma: 3evenMustBeEqual} to see that $(\mathscr{D},\mu)$ is not free.
	For the remaining multiplicities $\mu$ we give an extension of $(\mathscr{D},\mu)$ which is locally free along $H_4$, but has a localization isomorphic to $\mathscr{D}$. So $(\mathscr{D},\mu)$ cannot be free, by Corollary \ref{Cor: LocalFreeChaPolFactor}. We now give an alternative extension for all remaining multiplicities (up to a permutation of coordinates) where $\mu\neq (2k,2k,2k,1,1,1)$. Note that the local freeness along $H_4$ in the cases where \eqref{eq:star} is not met is given by Lemma \ref{lemma: localFreenessWithoutStar}. We write $H_{i4}=\ker(x_i-x_4)$ and $H_{ij4}=\ker(x_i+x_j-x_4)$. 

Let $\CA=\CE_\sigma(\mathscr{D},\mu)$ be an extension of $(\mathscr{D},\mu)$, where $\sigma$ is given as in the following table.
	\begin{center}
		\begin{tabular}{|c|c|c|c|c|c|c|c|}
			\hline
			$a$ & $b$ & $c$ & $d$ & $e$ & $f$ & $\sigma$ & Type   \\
			\hline
			even & even & odd & odd & odd & $\geq 1$  & $(0,0,0,1,1,0)$ & $(a)$ \\
			\hline
			even & even & odd & odd & even & $\geq 1$  & $(0,0,0,1,0,0)$ & $(a)$ \\
			\hline
			even & even & even & $\geq 3$& $\geq 1$ & $\geq 1$ & $\sigma(H_{i})=0$ and $\sigma(H_{ij})=1\iff \mu(H_{ij})$ odd  & $(c)$ \\
			\hline
			even & even & even & $1$ & $\geq 2$ & $\geq 2$ & $\sigma(H_{i})=0$ and $\sigma(H_{ij})=1\iff \mu(H_{ij})$ odd  & $(b)$ \\
			\hline
			even & even & even & $1$ & $1$ & $\geq 2$ & $\sigma(H_{i})=0$ and $\sigma(H_{ij})=1\iff \mu(H_{ij})$ odd  & $(a)$ \\
			\hline
		\end{tabular}
	\end{center}
The localizations of $\CA=\CE_\sigma(\mathscr{D},\mu)$ isomorphic to $\mathscr{D}$ are:
	\begin{enumerate}
		\item[(a)] $\CA_X=\{H_2,H_3,H_{14},H_{23},H_{124},H_{134}\}$ for $X=(H_2\cap H_3\cap H_{14})\in L(\CA)$,
		\item[(b)] $\CA_X=\{H_1,H_3,H_{24},H_{13},H_{124},H_{234}\}$ for $X=(H_1\cap H_3\cap H_{24})\in L(\CA)$,
		\item[(c)] $\CA_X=\{H_{1},H_{2},H_{34},H_{12},H_{134},H_{234}\}$ for $X=(H_1\cap H_2\cap H_{34})\in L(\CA)$.
	\end{enumerate}

We now show how the table above can be applied to one of the remaining multiplicities $\mu$ to derive the non-freeness for $(\mathscr{D},\mu)$. So let $\mu=(2,4,4,1,1,3)$. First note that since in this case $a,b,c$ are even and $d=e=1$, we use the last row of the table. The second to last column tells us that $\sigma=(0,0,0,1,1,1)$, since all of the values $\mu(H_{ij})$ are odd. Using the notation of Lemma \ref{lemma: localFreenessWithoutStar} and choosing an arbitrary $X\in L(\mathscr{D})_2$ with $\vert\mathscr{D}_X\vert=3$, we have parities $(even,even,odd)$ on $\mathscr{D}_X$ and $\sigma\vert_X=(0,0,1)$. So Lemma \ref{lemma: localFreenessWithoutStar} gives us the local freeness of $\CE_\sigma(\mathscr{D},\mu)$ along $H_4$. Since $\CE_\sigma(\mathscr{D},\mu)$ has localization $(a)$ with hyperplanes $$\{\ker(x_2),\ker(x_3),\ker(x_1-x_4),\ker(x_2+x_3),\ker(x_1+x_2-x_4),\ker(x_1+x_3-x_4)\},$$ (which is isomorphic to $\mathscr{D}$ and therefore non-free) we can apply Corollary \ref{Cor: LocalFreeChaPolFactor} to show that $(\mathscr{D},\mu)$ is not free.
	This completes the proof of the theorem.
\end{proof}

Deleting a hyperplane from $\CA_{C_3}$ results in either $\CA_{P_3}=\CA(A_3)$ or in $\mathscr{D}$. All free multiplicities on the reflection arrangement $\CA(A_3)$ have been classified, see \cite{abeteraowakefield:euler}, \cite{abenuidanumata:signedeliminable}, \cite{dipasquale: free A3 classification}. Since all free multiplicities are known on both types of deletions on $\CA_{C_3}$, the next step is to classify all free multiplicities on $\CA_{C_3}$ itself. 

\subsection{Some free multiplicities on $\CA_{C_3}$}
In this final section we present some free multiplicities on $\CA_{C_3}$ that we found by starting with a free multiplicity on $\CA_{P_3}$ or $\mathscr{D}$ and adding a hyperplane via Theorem $\ref{thm:add-del}$ to create a free multiplicity on  $\CA_{C_3}$. It turns out that there is only one free multiplicity on $\mathscr{D}$ that extends to a free multiplicity on $\CA_{C_3}$ in such a way. 

\begin{corollary}
Let $H=H_{123}$ and $$(\CA_{C_3},\mu)=x_1^{2k}x_2^{2k}x_3^{2k}(x_1+x_2)^1(x_1+x_3)^1(x_2+x_3)^1(x_1+x_2+x_3)^1.$$ 
Then $(\CA_{C_3},\mu)$ is free if and only if $k=1$.
\end{corollary}

\begin{proof}
Let $H=H_{123}$. Owing to  Theorem \ref{Theorem: AllDMultiplicities}, we see that $(\CA_{C_3}',\mu')$ is free with exponents $(2k+1,2k+1,2k+1)$. 
\noindent Now restrict $(\CA_{C_3},\mu)$ to $H$ and calculate $\mu^*$ by using Theorem \ref{WakaTheo} on the following rank $2$ flats:
$$\{H_{123},H_{12},H_{3}\},\{H_{123},H_{2},H_{13}\},\{H_{123},H_{1},H_{23}\}.$$
This gives $\mu^*=(2k,2k,2k)$ and therefore (again by Theorem \ref{WakaTheo}) $(\CA_{C_3}^H,\mu^*)$ is free with $\exp(\CA_{C_3}^H,\mu^*)=(3k,3k)$. So by Remark \ref{rem:AddDelThm_multi_simple} $(\CA_{C_3},\mu)$ is free if and only if $k=1$.
\end{proof}

\begin{defn}[{\cite[Def.~5.8]{abeteraowakefield:euler}}]
\label{def:ss}
    An arrangement $\CA$ is supersolvable if there exists a filtration
$\CA = \CA_r \supset \CA_{r-1}\supset\dots\supset \CA_2 \supset \CA_1$
such that:
\begin{itemize}
    \item $\rank(\CA_i ) = i$ for  $i = 1, \dots , r$;
    \item for any $H, H'\in \CA_i$, there exists some $H''\in \CA_{i-1}$ such that $H\cap H'\subset H''$.
\end{itemize} 
\end{defn}

\begin{theorem}[{\cite[Thm.~5.10]{abeteraowakefield:euler}}]\label{theorem: free vertex}
Suppose $\CA$ is supersolvable with a filtration $\CA=\CA_r\supset\CA_{r-1}\supset\dots\supset\CA_2\supset\CA_1$ and $r\geq 2$
as in Definition \ref{def:ss}.
Let $\mu$ be a multiplicity on $\CA$ and let $\mu_i = \mu\vert_{\CA_i}$. Let $\exp(\CA_2,\mu_2)=(e_1,e_2,0,\dots,0)$. Suppose that for each $H'\in \CA_d\backslash \CA_{d-1}$, $H''\in\CA_{d-1}$ $(d=3,\dots,r)$ and $X:=H'\cap H''$, we either have $$\CA_X=\{H',H''\}$$ or 
$$\mu(H'')\geq\left(\sum_{X\subset H\in \CA_d\backslash\CA_{d-1}}\mu(H)\right)-1.$$ Then $(\CA,\mu)$ is inductively free with $\exp(\CA,\mu)=(e_1,e_2,\vert\mu_3\vert-\vert\mu_2\vert,\dots,\vert\mu_r\vert-\vert\mu_{r-1}\vert,0,\dots,0)$.
\end{theorem}

We now use Theorem \ref{theorem: free vertex} to derive some 
classes of multiplicities on $\CA_{C_3}$.

\begin{proposition}
	\label{prop:SpecialFreeMultC_3}
	Let $(\CA_{C_3},\mu)$ be the multiarrangement $$Q(\CA_{C_3},\mu)=x_1^ax_2^bx_3^c(x_1+x_2)^d(x_1+x_3)^e(x_2+x_3)^f(x_1+x_2+x_3)^g$$ with multiplicity $\mu=(a,b,c,d,e,f,g), (d\geq e\geq f)$. Suppose that
	\begin{enumerate}
		\item[(i)] $\mu=(k,k,k,r,1,1,k)$, where $1\leq k \leq 3$ arbitrary and $r\geq 1$, or
		\item[(ii)] $\mu=(k,k,k,r,1,1,k)$, where $k>3$ arbitrary and $r\geq 
  2k-5$.
	\end{enumerate}
	Then $(\CA_{C_3},\mu)$ is free. Moreover, if $r\geq 2k$, then $(\CA_{C_3},\mu)$ is  inductively free with exponents $(r,2k+1,2k+1)$.
    For $\mu = (3,3,3,1,1,1,3)$, $(\CA_{C_3},\mu)$ is not inductively free.
\end{proposition}
\begin{proof}
Let $\mu(H_1)=\mu(H_2)=\mu(H_3)=\mu(H_{123})=k$ and $\mu(H_{12})=r\geq 2k-1,\mu(H_{13})=\mu(H_{23})=1$ and fix $H=H_{23}$, so $\mu'=(k,k,k,r,1,0,k)$. We choose the following supersolvable filtration for $(\CA_{C_3}',\mu')$:
$$\{H_1\}\subset\{H_1,H_2,H_{12}\}\subset\{H_1,H_2,H_3,H_{12},H_{13},H_{123}\}.$$
This filtration satisfies the requirements of Theorem \ref{theorem: free vertex}, because 
$\mu'(H_{1})\geq \mu'(H_{3})+\mu'(H_{13})-1\iff k\geq k$ and $\mu'(H_{2})\geq \mu'(H_{13})+\mu'(H_{123})-1\iff k\geq k$ and
$\mu'(H_{12})\geq \mu'(H_{3})+\mu'(H_{123})-1\iff 2k-1\geq 2k-1$. So $(\CA_{C_3}',\mu')$ is inductively free with exponents $(r,2k,2k+1)$. Restricting $(\CA_{C_3},\mu)$ to $H_{23}$ gives hyperplanes in correspondence to the rank $2$ flats $$\{H,H_2,H_3\},\{H,H_1,H_{123}\},\{H,H_{12}\},\{H,H_{13}\}$$ with $\mu^*=(k,k,r,1)$.
If $r\geq 2k$, then the restriction is free with exponents $(2k+1,r)$ since deleting the hyperplane with multiplicity $1$ gives a free arrangement with exponents $(2k,r)$ and using Theorem \ref{thm:add-del} shows that $(\CA_{C_3},\mu)$ is inductively free with exponents $(r,2k+1,2k+1)$.
Deleting $H_{12}$ with Theorem \ref{thm:add-del} and Lemma \ref{lemma: special B2 multiplicities} until $r=2k-5$ (and we end up at exponents $(2k-1,2k-1,2k-1)$) shows the other claim. In particular: If $r=3$, then $\mu=(3,3,3,1,1,1,3)$ and $(\CA_{C_3},\mu)$ is free with exponents $\exp(\CA_{C_3},\mu)=(5,5,5)$. Restricting to an arbitrary $H\in\CA_{C_3}$ results in $\vert(\CA_{C_3},\mu^*)\vert=8$. So Remark \ref{rem:AddDelThm_multi_simple} shows that $(\CA_{C_3}',\mu')$ can never be free and therefore $(\CA_{C_3},\mu)$ is neither inductively nor additively free. 
\end{proof}

    It can be shown that $\mu=(3,3,3,1,1,1,3)$ is the minimal multiplicity (with respect to $\vert\mu\vert$) on $\CA_{C_3}$ that is free but not inductively free.

{\bf Acknowledgments}: The authors acknowledge the financial support of the DFG-priority
program SPP 2458 ``Combinatorial Synergies'' (Grant \#RO 1072/25-1 (project number: 539865068) to G.~R\"ohrle). We are grateful to M.\ DiPasquale, T.~Hoge, and M.\ Wakefield for some helpful comments.
We would also like to thank the anonymous referee for pointing out some improvements in the exposition.

\newcommand{\etalchar}[1]{$^{#1}$}
\providecommand{\bysame}{\leavevmode\hbox to3em{\hrulefill}\thinspace}
\providecommand{\MR}{\relax\ifhmode\unskip\space\fi MR }
\bibliographystyle{plain}

\clearpage
\appendix
\section{Induction and Recursion tables}
\label{apdx:TablesA}


\begin{table}[H]
	\renewcommand{\arraystretch}{0.8}
	\begin{tabular}{lll}
		\hline
		$\exp(\CA,\mu)$ & $\alpha_H$ & $\exp(\CA'',\mu^*)$ \\
		\hline
		\hline
		$\exp(\CA_{G_1},\mu_1)=(4,5,5,5)$ & $H_{123}$ & $(4,5,5)$ \\
		$(4,4,5,5)$ & $H_{34}$ & $(4,4,5)$ \\
		$(4,4,4,5)$ & $H_{12}$ & $(4,4,5)$ \\
		$(3,4,4,5)$ & $H_{13}$ & $(3,4,5)$ \\
		$(3,3,4,5)$ & $H_{134}$ & $(3,3,4)$ \\
		$(3,3,4,4)$ & $H_{234}$ & $(3,3,4)$ \\
		$(3,3,3,4)$ & $H_{3}$ & $(3,3,3)$ \\
		$(3,3,3,3)$ & $H_{3}$ & $(3,3,3)$ \\
		$(2,3,3,3)$ & $H_{1234}$ & $(2,3,3)$ \\
		$(2,2,3,3)$ & $H_{124}$ & $(2,2,3)$ \\
		$(2,2,2,3)$ & $H_{4}$ & $(2,2,2)$ \\
		$(2,2,2,2)$ & $H_{1234}$ & $(2,2,2)$ \\
		$(1,2,2,2)$ & $H_{2}$ & $(1,2,2)$ \\
		$(1,1,2,2)$ & $H_{14}$ & $(1,1,2)$ \\
		$(1,1,1,2)$ & $H_{2}$ & $(1,1,2)$ \\
		$(0,1,1,2)$ & $H_{23}$ & $(0,1,2)$ \\
		$(0,0,1,2)$ & $H_{1}$ & $(0,0,1)$ \\
		$(0,0,1,1)$ & $H_{1}$ & $(0,0,1)$ \\
		$(0,0,0,1)$ & $H_{14}$ & $(0,0,0)$\\
		\hline
	\end{tabular}\caption{Inductive chain for $(\CA_{G_1},\mu_1)$ for $\mu=(2, 2, 2, 1, 2, 1, 1, 1, 2, 1, 1, 1, 1, 1)$.}\label{G1IndFree}	
\end{table}



\begin{table}[H]
	\renewcommand{\arraystretch}{0.7}
	\begin{tabular}{lll}
		\hline
		$\exp(\CA,\mu)$ & $H$ & $\exp(\CA'',\mu^*)$ \\
		\hline
		\hline
		$\exp(\CA_{C_3},\mu)=(7,7,7)$ & $H_{13}$ & $(7,7)$ \\
		$(6,7,7)$ & $H_{13}$ & $(7,7)$ \\
		$(5,7,7)$ & $H_{12}$ & $(5,7)$ \\
		$(5,6,7)$ & $H_{2}$ & $(5,6)$ \\
		$(5,6,6)$ & $H_{23}$ & $(5,6)$ \\
		$(5,5,6)$ & $H_{1}$ & $(5,5)$ \\
		$(5,5,5)$ & $H_{12}$ & $(5,5)$ \\
		$(4,5,5)$ & $H_{13}$ & $(4,5)$ \\
		$(4,4,5)$ & $H_{123}$ & $(4,4)$ \\
		$(4,4,4)$ & $H_{1}$ & $(4,4)$ \\
		$(3,4,4)$ & $H_{2}$ & $(3,4)$ \\
		$(3,3,4)$ & $H_{2}$ & $(3,3)$ \\
		$(3,3,3)$ & $H_{1}$ & $(3,3)$ \\
		$(2,3,3)$ & $H_{12}$ & $(2,3)$ \\
		$(2,2,3)$ & $H_{3}$ & $(2,2)$ \\
		$(2,2,2)$ & $H_{23}$ & $(2,2)$ \\
		$(1,2,2)$ & $H_{3}$ & $(1,2)$ \\
		$(1,1,2)$ & $H_{23}$ & $(1,2)$ \\
		$(0,1,2)$ & $H_{123}$ & $(0,1)$\\
		$(0,1,1)$ & $H_{3}$ & $(0,1)$ \\
		$(0,0,1)$ & $H_{123}$ & $(0,0)$\\
		\hline 
	\end{tabular}\caption{Inductive chain for $(\CA_{C_3},\mu)$ for $\mu\equiv 3$.}\label{C3ThreeIndFree}
\end{table}


\begin{table}
	\centering
	\renewcommand{\arraystretch}{0.7}
	\begin{tabular}{llll}
		\hline
		Addition/Deletion & $\exp(\CA,\mu)$ & $H$ & $\exp(\CA'',\mu^*)$ \\
		\hline
		\hline
		Deletion & $\exp(\CA_{G_2},\mu_2)=(5,5,5,5)$ & $H_{234}$ & $(5,5,5)$ \\
		Addition & $(5,5,5,6)$ & $H_{14}$ & $(5,5,5)$ \\
		Addition & $(5,5,5,5)$ & $H_{1}$ & $(5,5,5)$ \\
		Addition & $(4,5,5,5)$ & $H_{134}$ & $(4,5,5)$ \\
		Addition & $(4,4,5,5)$ & $H_{1234}$ & $(4,4,5)$ \\
		Addition & $(4,4,4,5)$ & $H_{4}$ & $(4,4,4)$ \\
		Addition & $(4,4,4,4)$ & $H_{124}$ & $(4,4,4)$ \\
		Addition & $(3,4,4,4)$ & $H_{2}$ & $(3,4,4)$ \\
		Addition & $(3,3,4,4)$ & $H_{23}$ & $(3,3,4)$ \\
		Addition & $(3,3,3,4)$ & $H_{1}$ & $(3,3,3)$ \\
		Addition & $(3,3,3,3)$ & $H_{3}$ & $(3,3,3)$ \\
		Addition & $(2,3,3,3)$ & $H_{13}$ & $(2,3,3)$ \\
		Addition & $(2,2,3,3)$ & $H_{123}$ & $(2,2,3)$ \\
		Addition & $(2,2,2,3)$ & $H_{234}$ & $(2,2,3)$ \\
		Addition & $(1,2,2,3)$ & $H_{3}$ & $(1,2,2)$ \\
		Addition & $(1,2,2,2)$ & $H_{12}$ & $(1,2,2)$ \\
		Addition & $(1,1,2,2)$ & $H_{34}$ & $(1,1,2)$ \\
		Addition & $(1,1,1,2)$ & $H_{2}$ & $(1,1,1)$ \\
		Addition & $(1,1,1,1)$ & $H_{24}$ & $(1,1,1)$ \\
		Addition & $(0,1,1,1)$ & $H_{4}$ & $(0,1,1)$ \\
		Addition & $(0,0,1,1)$ & $H_{1234}$ & $(0,0,1)$ \\
		Addition & $(0,0,0,1)$ & $H_{234}$ & $(0,0,0)$ \\
		\hline
	\end{tabular}\caption{Recursive chain for $(\CA_{G_2},\mu_2)$.}\label{tableRecFree}
\end{table}



\clearpage

\end{document}